\documentclass[reqno]{amsart}
\usepackage{amssymb,amsmath,hyperref}
\usepackage{amsrefs, float}
\usepackage[foot]{amsaddr}
\usepackage{bbold,stackrel, multicol}
\usepackage{mathrsfs}

\newcommand{\Z}{{\textsf{\textup{Z}}}}

\newtheorem{thm}{Theorem}

\newtheorem{cor}[thm]{Corollary}
\newtheorem{defi}[thm]{Definition}

\newtheorem{nota}[thm]{Notation}

\newtheorem{princ}[thm]{Principle}

\newtheorem{ack}[thm]{Acknowledgement}

\newtheorem*{tempo*}{Template}

\newcommand\be{\begin{equation}}
\newcommand\ee{\end{equation}} 

\usepackage{amsmath,amsfonts} 

\usepackage[applemac]{inputenc}

\def\bdefi{\begin{defi}}
\def\edefi{\end{defi}}
\def\bnota{\begin{nota}\rm}
\def\enota{\end{nota}}
\def\FIVE{\Pi_{1}^{1}\text{-\textup{\textsf{CA}}}_{0}}
\def\SIX{\Pi_{2}^{1}\text{-\textsf{\textup{CA}}}_{0}}

\def\ATR{\textup{\textsf{ATR}}}

\def\ZF{\textup{\textsf{ZF}}}

\def\osc{\textup{\textsf{osc}}}

\def\RCA{\textup{\textsf{RCA}}}
\def\({\textup{(}}
\def\){\textup{)}}

\def\RCAo{\textup{\textsf{RCA}}_{0}^{\omega}}
\def\ACAo{\textup{\textsf{ACA}}_{0}^{\omega}}

\def\WKL{\textup{\textsf{WKL}}}

\def\bye{\end{document}}
\def\N{{\mathbb  N}}
\def\Q{{\mathbb  Q}}
\def\R{{\mathbb  R}}

\def\SS{\textup{\textsf{S}}}

\def\di{\rightarrow}

\def\asa{\leftrightarrow}

\def\ACA{\textup{\textsf{ACA}}}

\def\QFAC{\textup{\textsf{QF-AC}}}
\def\AC{\textup{\textsf{AC}}}

\def\osc{\textup{\textsf{osc}}}

\def\FC{\textup{\textsf{Finite Choice}}}
\def\Hyp{\textup{\textsf{HYP}}}
\def\ABW{\textup{\textsf{ABW}}}
\def\CLC{\textup{\textsf{ClC}}}

\def\WSAC{\textup{\textsf{weak-$\Sigma_{1}^{1}$-AC$_{0}$}}}
\def\USAC{\textup{\textsf{weak-$\Sigma_{1}^{1}$-AC$_{0}$}}}
\def\SAC{\textup{\textsf{$\Sigma_{1}^{1}$-AC$_{0}$}}}
\def\FSAC{\textup{\textsf{finite-$\Sigma_{1}^{1}$-AC$_{0}$}}}
\def\WFSAC{\textup{\textsf{height-$\Sigma_{1}^{1}$-AC$_{0}$}}}

\def\cocode{\textup{\textsf{cocode}}}

\def\IND{\textup{\textsf{IND}}}

\def\fin{\textup{\textsf{fin}}}

\def\eps{\varepsilon}

\def\HYP{\textup{\textsf{HYP}}}

\usepackage{graphicx}
\usepackage{tikz}
\usetikzlibrary{matrix}
\usepackage{comment,tikz-cd}

\numberwithin{equation}{section}
\numberwithin{thm}{section}

\begin{document}
\title{Connecting real and hyperarithmetical analysis}
\author{Sam Sanders}
\address{Department of Philosophy II, RUB Bochum, Germany}
\email{sasander@me.com}
\keywords{Higher-order arithmetic, hyperarithmetical analysis}
\subjclass[2010]{03B30, 03F35}
\begin{abstract}
Going back to Kreisel in the Sixties, \emph{hyperarithmetical analysis} is a cluster of logical systems just beyond arithmetical comprehension.  Only recently natural examples of theorems from the mathematical mainstream were identified that fit this category.
In this paper, we provide \emph{many} examples of theorems of real analysis that \emph{sit within the range of hyperarithmetical analysis}, namely between the higher-order version of 
$\Sigma_{1}^{1}$-$\AC_{0}$ and $\textsf{weak}$-$\Sigma_{1}^{1}$-$\AC_{0}$, working in Kohlenbach's higher-order framework.  Our example theorems are based on the \emph{Jordan decomposition theorem}, \emph{unordered sums}, \emph{metric spaces}, and \emph{semi-continuous functions}.   Along the way, we identify a couple of new systems of hyperarithmetical analysis.  
\end{abstract}


\maketitle
\thispagestyle{empty}

\section{Introduction}\label{intro}
\subsection{Motivation and overview}
The aim of this paper is to exhibit many natural examples of theorems from real analysis that exist \textbf{in the range of hyperarithmetical analysis}.  
The exact meaning of `hyperarithmetical analysis' and the previous boldface text is discussed in Section~\ref{hyper}, but intuitively speaking the latter amounts to being sandwiched between known systems of hyperarithmetical analysis or their higher-order extensions.  We shall work in Kohlenbach's framework from \cite{kohlenbach2}, with which we assume basic familiarity.  

\smallskip

We introduce some necessary definitions and axioms in Section \ref{prelim}.  We shall establish that the following inhabit the range of hyperarithmetical analysis. 
\begin{itemize}
\item Basic properties of (Lipschitz) continuous functions on compact metric spaces \emph{without second-order representation/separability conditions}, including  the generalised intermediate value theorem  (Section \ref{disorder}).
\item Properties of functions of \emph{bounded variation}, including the \emph{Jordan decomposition theorem}, where the total variation is given (Section~\ref{BVS}).  
\item Properties of semi-continuous functions and closed sets (Section~\ref{evelinenef}).
\item Convergence properties of \emph{unordered sums} (Section \ref{unorder}).  
\end{itemize}
These results still go through if we restrict to arithmetically defined objects by Theorem \ref{slonk}. 
To pinpoint the exact location of the aforementioned principles, we introduce a new `finite choice' principle based on $\FSAC$ from \cite{gohzeg} (see Section~\ref{prelim}), 
using Borel's notion of \emph{height function} (\cite{opborrelen4, opborrelen5}). 

\smallskip

Finally, as to conceptual motivation, the historical examples of systems of hyperarithmetical analysis are rather logical in nature and natural examples from the mathematical mainstream are a relatively recent discovery, as discussed in Section~\ref{hyper}. Our motivation is to show that third-order
arithmetic exhibits \emph{many} robust examples of theorems in the range of hyperarithmetical analysis, similar perhaps to how so-called splittings and disjunctions are much more plentiful in third-order arithmetic, as explored in  \cite{samsplit}. 
In this paper, we merely develop certain examples and indicate the many possible variations.  

\subsection{Preliminaries}\label{prelim}
We introduce some basic definitions and axioms necessary for this paper. 
We note that subsets of $\R$ are given by their characteristic functions as in Definition \ref{char}, well-known from measure and probability theory.
We shall generally work over $\ACAo$ -defined right below- as some definitions make little sense over the base theory $\RCAo$.  We refer to \cite{kohlenbach2} for the latter.

\smallskip

First of all, full second-order arithmetic $\Z_{2}$ is the `upper limit' of second-order RM.  The systems $\Z_{2}^{\omega}$ and $\Z_{2}^{\Omega}$ are conservative extensions of $\Z_{2}$ by \cite{hunterphd}*{Cor.\ 2.6}. 
The system $\Z_{2}^{\Omega}$ is $\RCAo$ plus Kleene's quantifier $(\exists^{3})$ (see e.g.\ \cite{dagsamXIV} or \cite{hunterphd}), while $\Z_{2}^{\omega}$ is $\RCAo$ plus $(\SS_{k}^{2})$ for every $k\geq 1$; the latter axiom states the existence of a functional $\SS_{k}^{2}$ deciding $\Pi_{k}^{1}$-formulas in Kleene normal form.  
The system $\FIVE^{\omega}\equiv \RCAo+(\SS_{1}^{2})$ is a $\Pi_{3}^{1}$-conservative extension of $\FIVE$ (\cite{yamayamaharehare}), where $\SS_{1}^{2}$ is also called the \emph{Suslin functional}. 
We also write $\ACAo$ for $\RCAo+(\exists^{2})$ where the latter is as follows 
\be\label{muk}\tag{$\exists^{2}$}
(\exists E:\N^{\N}\di \{0,1\})(\forall f \in \N^{\N})\big[(\exists n\in \N)(f(n)=0) \asa E(f)=0    \big].
\ee
The system $\ACAo$ is a conservative extension of $\ACA_{0}$ by \cite{hunterphd}*{Theorem 2.5}.
Over $\RCAo$, $(\exists^{2})$ is equivalent to $(\mu^{2})$, where the latter expresses the existence of Feferman's $\mu$ (see \cite{kohlenbach2}*{Prop.\ 3.9}), defined as follows for all $f\in \N^{\N}$:
\[
\mu(f):= 
\begin{cases}
n & \textup{if $n$ is the least natural number such that $f(n)=0$ }\\
0 & \textup{if $f(n)>0$ for all $n\in \N$}
\end{cases}
.
\]
The following schema is essential to our enterprise, as discussed in Section \ref{hyper}.  
\begin{princ}[$\QFAC^{0,1}$]
For any $ Y:\N^{\N}\di\N$, if $(\forall n\in \N)(\exists f\in \N^{\N})(Y(f, n)=0)$, then there exists a sequence $  (f_{n})_{n\in \N}$ in $\N^{\N}$ with $(\forall n\in \N)(Y(f_{n}, n)=0)$.
\end{princ}
The local equivalence between sequential and `epsilon-delta' continuity cannot be proved in $\ZF$, but can be established in $\RCAo+\QFAC^{0,1}$ (see \cite{kohlenbach2}).
Thus, it should not be a surprise that the latter system is often used as a base theory too.

\smallskip

Secondly, we make use the following standard definitions concerning sets.  
\bdefi[Sets]\label{char}~
\begin{itemize}
\item A subset $A\subset \R$ is given by its characteristic function $F_{A}:\R\di \{0,1\}$, i.e.\ we write $x\in A$ for $ F_{A}(x)=1$, for any $x\in \R$.
\item A set $A\subset \R$ is \emph{enumerable} if there is a sequence of reals that includes all elements of $A$.
\item A set $A\subset \R$ is \emph{countable} if there is $Y: \R\di \N$ that is injective on $A$, i.e.\
\[
(\forall x, y\in A)( Y(x)=_{0}Y(y)\di x=_{\R}y).  
\]
\item A set $A\subset \R$ is \emph{strongly countable} if there is $Y: \R\di \N$ that is injective and surjective on $A$; the latter means that $(\forall n\in \N)(\exists x\in A)(Y(x)=n)$.
\item A set $A\subset \R$ is \emph{finite} in case there is $N\in \N$ such that for any finite sequence $(x_{0}, \dots, x_{N})$, there is $i\leq N$ with $x_{i}\not \in A$.  
We sometimes write `$|A|\leq N$'.
\end{itemize}
\edefi
Thirdly, we list the following second-order system needed below.  
\begin{princ}[$\FSAC$, \cite{gohzeg}]
The system $\RCA_{0}$ plus for any arithmetical $\varphi$:
\[
(\forall n\in \N)(\exists \textup{ nonzero finitely many } X\subset \N )\varphi(n, X ) \di (\exists (X_{n})_{n\in \N} )(\forall n\in \N)\varphi(n, X_{n} ), 
\]
where `$(\exists \textup{ nonzero finitely many } X\subset \N )\varphi(n, X )$' means that there is a non-empty sequence $(X_{0}, \dots, X_{k})$ such that for any $X\subset \N$,
$\varphi(n, X)\asa (\exists i\leq k)(X_{i}=X)$.
\end{princ}
We let $\WFSAC$  be $\FSAC$ where we additionally assume $g\in \N^{\N}$ to be given such that for all $n$, $g(n)\geq k+1$ where $k+1$ is the length of the sequence $(X_{0}, \dots, X_{k})$ in the formula `$(\exists \textup{ nonzero finitely many } X\subset \N )\varphi(n, X )$'.  
We have the following straightforward connections:
\[
\SAC\di \FSAC\di \WFSAC\di  \WSAC,  
\]
i.e.\ $\WFSAC$ is also a system of hyperarithmetical analysis by Section \ref{hyper}.
In the grand scheme of things, $g$ is a \emph{height function}, a notion that goes back to Borel (\cites{opborrelen3, opborrelen5}) and is studied in RM in \cite{samBIG, samBIG3}. 

\subsection{On hyperarithmetical analysis}\label{hyper}
Going back to Kreisel (\cite{kreide}), the notion of \emph{hyperarithmetical set} (see e.g.\ \cite{simpson2}*{VIII.3}) gives rise to the second-order definition of \emph{theory/theorem of hyperarithmetical analyis} (THA for brevity, see e.g.\ \cite{skore3}).  In this section, we recall known results regarding THAs, including the exact (rather technical) definition, for completeness.

\smallskip

First of all, well-known THAs are $\Sigma_{1}^{1}$-$\textsf{CA}_{0}$ and $\textsf{weak}$-$\Sigma_{1}^{1}$-$\textsf{CA}_{0}$  (see \cite{simpson2}*{VII.6.1 and VIII.4.12}), where the latter is the former with the antecedent restricted to unique existence.  Any system \emph{between} two THAs is \emph{also} a THA, which is a convenient way of establishing that a given system is a THA.  

\smallskip

Secondly, at the higher-order level, $\ACAo+\QFAC^{0,1}$ from Section \ref{prelim} is a conservative extension of $\Sigma_{1}^{1}$-$\textsf{CA}_{0}$ by \cite{hunterphd}*{Cor.\ 2.7}.  
This is established by extending any model $\mathcal{M}$ of $\SAC$ to a model $\mathcal{N}$ of $\ACAo+\QFAC^{0,1}$, where the second-order part of $\mathcal{N}$ is isomorphic to $\mathcal{M}$.  
In this paper, we study (higher-order) systems that imply $\textsf{weak}$-$\Sigma_{1}^{1}$-$\textsf{CA}_{0}$ and are implied by $\ACAo+\QFAC^{0,1}$.  In light of the aforementioned conservation result, it is reasonable to refer to such intermediate third-order systems as \emph{existing in the range of hyperarithmetical analysis}.  

\smallskip

Thirdly, finding a \emph{natural} THA, i.e.\ hailing from the mathematical mainstream, is surprisingly hard. 
Montalb\'an's \textsf{INDEC} from \cite{monta2}, a special case of Jullien's \cite{juleke}*{IV.3.3}, is generally considered to be the first such statement.  The latter theorem by Jullien can be found in \cite{aardbei}*{6.3.4.(3)} and \cite{roosje}*{Lemma~10.3}.  
The monographs \cites{roosje, aardbei, juleke} are all `rather logical' in nature and $\textsf{INDEC}$ is the \emph{restriction} of a higher-order statement to countable linear orders in the sense of RM (\cite{simpson2}*{V.1.1}), i.e.\ such orders are given by sequences.  In \cite{dagsamXI}*{Remark~2.8} and \cite{samcount}*{Remark 7 and \S3.4}, a number of third-order statements are identified, including the Bolzano-Weierstrass theorem and K\"onig's infinity lemma, that are in the range of hyperarithmetical analysis.  Shore and others have studied a considerable number of THAs from graph theory \cites{skore1, skore2,gohzeg}.  A related concept is that of \emph{almost theorem/theory of hyperarithmetical analysis} (ATHA for brevity, \cite{skore3}), which is weaker than $\ACA_{0}$ but becomes a THA when combined with the latter.   

\smallskip

Finally, we consider the official definition of THA from \cite{monta2} based on $\omega$-models.
\bdefi\label{flagon}
A system $T$ of axioms of second-order arithmetic is a \emph{theory/theorem of hyperarithmetical analysis} in case
\begin{itemize}
\item  $T$ holds in $\HYP(Y)$ for every $Y \subset \omega$, where $\Hyp(Y)$ is the $\omega$-model consisting of all sets hyperarithmetic in $Y$,
\item all $\omega$-models of $T$ are hyperarithmetically closed. 
\end{itemize}
\edefi
Here, an $\omega$-model is \emph{hyperarithmetically closed} if it is closed under disjoint union and for every set $X, Y \subset \omega$, if $X$ is hyperarithmetically reducible to $Y$ and $Y$ is in the model, then $X$ is in the model too.
In turn, this notion of reducibility means that $n\in X$ can be expressed by a $\Delta_{1}^{1}$-formula with $Y$ as a parameter; we refer to \cite{monta2}*{Theorem 1.14} for equivalent formulations.

\section{Metric spaces}\label{disorder}
We introduce the well-known definition of metric space $(M, d)$ to be used in this paper (Section \ref{bd}), where we always assume $M$ to be a subset of $\R$, up to coding of finite sequences.  
We show that basic properties of (Lipschitz) continuous functions on such metric spaces exist in the range of hyperarithmetical analysis (Section~\ref{zag}), 
even if we restrict to arithmetically defined objects (Theorem \ref{slonk}). 
We have previously studied metric spaces in \cite{sammetric}; to our own surprise, some of these results have nice generalisations relevant to the study of hyperarithmetical analysis.  
\subsection{Basic definitions}\label{bd}
We shall study metric spaces $(M, d)$ as in Definition~\ref{donkc}.  We assume that $M$ comes with its own equivalence relation `$=_{M}$' and that the metric $d$ satisfies 
the axiom of extensionality relative to `$=_{M}$' as follows:
\[
(\forall x, y, v, w\in M)\big([x=_{M}y\wedge v=_{M}w]\di d(x, v)=_{\R}d(y, w)\big).
\]
Similarly to functions on the reals, `$F:M\di \R$' denotes a function from $M$ to the reals that satisfies the following instance of the axiom of function extensionality:
\be\tag{\textup{\textsf{E}}$_{M}$}\label{koooooo}
(\forall x, y\in M)(x=_{M}y\di F(x)=_{\R}F(y)).
\ee
We recall that the study of metric space in second-order RM is at its core based on equivalence relations, as discussed explicitly in e.g.\ \cite{simpson2}*{I.4} or \cite{damurm}*{\S10.1}.
\bdefi\label{donkc}
A functional $d: M^{2}\di \R$ is a \emph{metric on $M$} if it satisfies the following properties for $x, y, z\in M$:
\begin{enumerate}
 \renewcommand{\theenumi}{\alph{enumi}}
\item $d(x, y)=_{\R}0 \asa  x=_{M}y$,
\item $0\leq_{\R} d(x, y)=_{\R}d(y, x), $
\item $d(x, y)\leq_{\R} d(x, z)+ d(z, y)$.
\end{enumerate}
\edefi
We shall only study metric spaces $(M, d)$ with $M\subset \N^{\N}$ or $M\subset \R$. 
To be absolutely clear, quantifying over $M$ amounts to quantifying over $\N^{\N}$ or $\R$, perhaps modulo coding of finite sequences, i.e.\ the previous definition can be made in third-order arithmetic for the intents and purposes of this paper.   Since we shall study compact metric spaces, this restriction is minimal in light of \cite{buko}*{Theorem 3.13}.

\smallskip

Sub-sets of $M$ are defined via characteristic functions, like for the reals in Definition \ref{char}, keeping in mind \eqref{koooooo}.  
In particular, we use standard notation like $B_{d}^{M}(x, r)$ to denote the open ball $\{y\in M: d(x, y)<_{\R}r\}$.  

\smallskip

Secondly, the following definitions are now standard, where we note that a different nomenclature is sometimes used in second-order RM.
A sequence $(w_{n})_{n\in \N}$ in $(M, d)$ is \emph{Cauchy} if $(\forall k\in \N)(\exists N\in \N)(\forall m, n\geq N)( d(w_{n}, w_{m})<\frac{1}{2^{k}})$
\bdefi[Compactness and around]\label{deaco} 
For a metric space $(M, d)$, we say that
\begin{itemize}
\item $(M, d)$ is \emph{weakly countably-compact} if for any $(a_{n})_{n\in \N}$ in $M$ and sequence of rationals $(r_{n})_{n\in \N}$ such that we have $M\subset \cup_{n\in \N}B^{M}_{d}(a_{n}, r_{n})$, there is $m\in \N$ such that  $M\subset \cup_{n\leq m}B^{M}_{d}(a_{n}, r_{n})$,
\item $(M, d)$ is \emph{countably-compact} if for any sequence $(O_{n})_{n\in \N}$ of open sets in $M$ such that $M\subset \cup_{n\in \N}O_{n}$, there is $m\in \N$ such that  $M\subset \cup_{n\leq m}O_{n}$,
\item $(M, d)$ is \emph{compact} in case for any $\Psi:M\di \R^{+}$, there are $x_{0}, \dots, x_{k}\in M$ such that $\cup_{i\leq k}B_{d}^{M}(x_{i}, \Psi(x_{i}))$ covers $M$, 
\item $(M, d)$ is \emph{sequentially compact} if any sequence has a convergent sub-sequence,
\item $(M, d)$ is \emph{limit point compact} if any infinite set in $M$ has a limit point,
\item $(M, d)$ is \emph{complete} in case every Cauchy sequence converges,
\item $(M, d)$ is \emph{totally bounded} if for all $k\in \N$, there are $w_{0}, \dots, w_{m}\in M$ such that $\cup_{i\leq m}B_{d}^{M}(w_{i}, \frac{1}{2^{k}})$ covers $M$.
\item a function $f:M\di \R$ is \emph{topologically continuous} if for any open $V\subset \R$, the set $f^{-1}(V)=\{x\in M: f(x)\in V\}$ is also open. 
\item a function $f:M\di \R$ is \emph{closed} if for any closed $C\subset M$, we have that $f(C)$ is also closed. \(\cites{munkies, ooskelly, leelee, searinghot}\).  
\end{itemize}
\edefi
\noindent
Regarding the final item, the set $f(C)$ does not necessarily exist in $\ACAo$, but `$f(C)$ is closed' makes sense\footnote{In particular, `$y\in f(C)$' means `$(\exists x\in C)(f(x)=y)$' and `$f(C)$ is closed' means `$(\forall y\not\in f(C))( \exists N\in \N )(\forall z\in B(z, \frac{1}{2^{N}}))( z\not\in f(C) )$', as expected.} as shorthand for the associated well-known definition.
We could study other notions, e.g.\ the Lindel\"of property or compactness based on nets, but have opted to stick to basic constructs already studied in second-order RM. 

\smallskip

Finally, fragments of the induction axiom are sometimes used, even in an essential way, in second-order RM (see e.g.\ \cite{neeman,skore3}).  
The equivalence between induction and bounded comprehension is also well-known in second-order RM (\cite{simpson2}*{X.4.4}).
We shall need a little bit of the induction axiom as follows.
\begin{princ}[$\IND_{2}$]\label{IX}
Let $Y^{2}$ satisfy $(\forall n \in \N) (\exists f \in 2^{\N})[Y(n,f)=0]$.  Then $ (\forall n\in \N)(\exists w^{1^{*}})\big[ |w|=n\wedge  (\forall i < n)(Y(i,w(i))=0)\big]$.  
\end{princ}
We let $\IND_{0}$ and $\IND_{1}$ be $\IND_{2}$ with `$(\exists f \in 2^{\N})$' restricted to respectively `$(\exists \textup{ at most one } f \in 2^{\N})$' and `$(\exists! f \in 2^{\N})$'.
We have previously used $\IND_{i}$ for $i=0, 1, 2 $ in the RM of the Jordan decomposition theorem (\cite{dagsamXI}).
By the proof of \cite{dagsamXI}*{Theorem 2.16}, $\Z_{2}^{\omega}+\IND_{2}$ cannot prove the uncountability of the reals formulated as: \emph{the unit interval is not strongly countable}.  

\subsection{Metric spaces and hyperarithmetical analysis}\label{zag}
\subsubsection{Introduction}
In this section, we identify a number of the basic properties of metric spaces in the range of hyperarithmetical analysis, as listed on the next page.  
The Axiom of Choice for finite sets as in Principle \ref{FCC} naturally comes to the fore. Clearly, the principle $\FC$ implies $\FSAC$ over $\ACAo$.
\begin{princ}[$\FC$]\label{FCC}
Let $(X_{n})_{n\in \N}$ be a sequence of non-empty finite sets in $[0,1]$.  Then there is $(x_{n})_{n\in \N}$ such that $x_n\in X_{n}$ for all $n\in \N$.
\end{princ}
\noindent
In more detail, we will establish that the following theorems are intermediate between $\ACAo+\QFAC^{0,1}$ and $\ACAo+\FC$. 
\begin{itemize}
\item Basic properties of continuous functions on \emph{sequentially} compact metric spaces (Section \ref{gawu}).
\item Basic properties of \emph{sequentially} continuous functions on (countably) compact metric spaces (Section \ref{gawu2}).
\item Restrictions of the previous results to arithmetically defined or \emph{Lipschitz} continuous functions (Section \ref{gawu3}).
\item Basic properties of \emph{connected} metric spaces, including the generalisation of the  \emph{intermediate value theorem} (Section \ref{gawu4}).
\end{itemize}
We sometimes obtain elegant equivalences, like for the intermediate value theorem (Theorem \ref{preppiu}).  We believe there is no `universal' approach 
to the previous results: each section is based on a very particular kind of metric space.  

\subsubsection{Sequentially compact spaces}\label{gawu}
In this section, we establish that basic properties of sequentially compact spaces inhabit the range of hyperarithmetical analysis. 
The following theorem is our first result, to be refined below. 
\begin{thm}[$\ACAo+\IND_{2}$]\label{klonz}
The principle $\FC$ follows from any of the items \eqref{itemaa}-\eqref{AA} where $(M, d)$ is any metric space with $M\subset \R$; the principle $\QFAC^{0,1}$ implies items \eqref{itemaa}-\eqref{wikked}.
\begin{enumerate}
\renewcommand{\theenumi}{\alph{enumi}}
\item For sequentially compact $(M, d)$, any continuous $f:M\di \R$ is bounded.\label{itemaa}
\item The previous item with `is bounded' replaced by `is uniformly continuous'.  
\label{itemgas}
\item For sequentially compact $(M, d)$ and continuous $f:M\di \R$ with $\inf_{x\in M}f(x)=y\in \R$ given, there is $x\in M$ with $f(x)=y$.\label{itemk}
\item \(Dini\).  Let $(M, d)$ be sequentially compact and let $f_{n}: (M\times \N)\di \R$ be a monotone sequence of continuous functions converging to continuous $f:M\di \R$.  
Then the convergence is uniform.   \label{diniitem}
\item For a sequentially compact metric space $(M, d)$, equicontinuity implies uniform equicontinuity \(\cite{magnus}*{Prop.\ 4.25}\).  \label{equi}
\item For a sequentially compact metric space $(M, d)$ with $M\subset [0,1]$ infinite, there is a discontinuous function $f:M\di \R$.\label{discu}
\item \(Closed map lemma, \cites{munkies, leelee, kura,mannetti}\) For a sequentially compact metric space $(M, d)$ any continuous function $f:M\di \R$ is closed.  \label{cml}
\item For sequentially compact $(M, d)$ and disjoint closed $C, D\subset M$, $d(C, D)>0$.\label{dizzy}
\item \(weak Cantor intersection theorem\) For a sequentially compact metric space $(M, d)$ and a sequence of closed sets with $M \supseteq C_{n}\supseteq C_{n+1}\ne\emptyset$, such that $\lim_{n\di \infty}\textsf{\textup{diam}}(C_{n})=0$, there is a unique $w\in \cap_{n\in \N}C_{n}$.\label{wikked} 
\item \(Ascoli-Arzel\`a\)  For sequentially compact $(M, d)$, a uniformly bounded equi-continuous sequence of functions on $M$ has a uniformly
convergent sub-sequence.\label{AA}
\end{enumerate}
The theorem still goes through if we require a modulus of continuity in item \eqref{itemaa} or if we replace `continuity' by `topological continuity' in 
items \eqref{itemaa}-\eqref{discu}.  
\end{thm}
\begin{proof}
We first derive $\FC$ from item \eqref{itemaa} via a proof-by-contradiction.  
To this end, fix a sequence of non-empty finite sets of reals $(X_{n})_{n\in \N}$.  Suppose there is no sequence $(x_{n})_{n\in \N}$ of reals such that $x_{n}\in X_{n}$ for all $n\in \N$. 
We now define
\be\label{MZ}
M_{0}:=\{ w^{1^{*}}: (\forall i<|w|)(w(i)\in X_{i})\}, 
\ee
where $w^{1^{*}}$ is a finite sequence of reals of length $|w|$, readily coded using $(\exists^{2})$.   
We define the equivalence relation `$=_{M_{0}}$' as follows: the relation $w=_{M_{0}}v$ holds if $|w|=_{0}|v|$, where $w, v\in M_{0}$.
The metric $d_{0}:M_{0}^{2}\di \R$ is defined as $d_{0}(w, v)=|\frac{1}{2^{|v|}}-\frac{1}{2^{|w|}}|$ for any $w, v\in M_{0}$.
We then have $d_{0}(v, w)=_{\R}0 \asa |v|=_{0}|w|\asa  v=_{M_{0}}w$ as required.  
We also have $0\leq d_{0}(v, w)=_{\R}d_{0}(w, v)$ for any $v, w\in M_{0}$, while for any $z\in M_{0}$ we observe:
\[\textstyle
d_{0}(v, w)=|\frac{1}{2^{|v|}}-\frac{1}{2^{|w|}}|= |\frac{1}{2^{|v|}}-\frac{1}{2^{|z|}}+\frac{1}{2^{|z|}}   -\frac{1}{2^{|w|}}|\leq |\frac{1}{2^{|v|}}-\frac{1}{2^{|z|}}|+|\frac{1}{2^{|z|}}   -\frac{1}{2^{|w|}}|=d_{0}(v, z)+d_{0}(z, w)
\]
by the triangle equality of the absolute value on the reals.  Hence, $(M_{0}, d_{0})$ is a metric space as in Definition \ref{donkc}.

\smallskip

To show that $(M_{0}, d_{0})$ is sequentially compact, let $(w_{n})_{n\in \N}$ be a sequence in $M_{0}$ and consider the following case distinction.  
In case $(\forall n\in \N)( |w_{n}|<m_{0})$ for some fixed $m_{0}\in \N$, then $(w_{n})_{n\in \N}$ contains at most $(m_{0}+1)!$ different
elements.  The pigeon hole principle now implies that at least one $w_{n_{0}}$ occurs infinitely often in $(w_{n})_{n\in \N}$, i.e.\ $(w_{n_{0}})_{n\in \N}$ is a convergent sub-sequence.  
In case $(\forall  m\in \N)(\exists n\in \N)( |w_{n}|\geq m)$, the sequence $(w_{n})_{n\in \N}$ yields a sequence $(x_{n})_{n\in \N}$ such that $x_{n}\in X_{n}$ for all $n\in \N$, which is impossible by assumption.
Hence, $(M_{0}, d_{0})$ is a sequentially compact metric space.  

\smallskip

Next, define $f:M_{0}\di \R$ as $f(w):=|w|$, which is clearly not bounded on $M_{0}$, which one shows using $\IND_{2}$.  
To show that $f$ is continuous at $w_{0}\in M_{0}$, consider the formula $|\frac{1}{2^{|w_{0}|}}-\frac{1}{2^{|v|}}|=d_{0}(v, w_{0})<\frac{1}{2^{N}}$; the latter is false for $N\geq |w_{0}|+2$ and any $v \ne_{M_{0}}w_{0}$.  Hence, the following formula is vacuously true: 
\be\label{coki}\textstyle
(\forall k\in \N)(\exists N\in \N)(\forall v\in B_{d_{0}}^{M_{0}}(w_{0}, \frac{1}{2^{N}}))(|f(w_{0})-f(v)|<_{\R} \frac{1}{2^{k}}).
\ee
i.e.\ $f$ is continuous at $w_{0}\in M_{0}$, with a modulus of continuity given by $h(w,k):= \frac{1}{2^{|w|+k+2}}$.
To see that $f$ is also topologically continuous,  fix an open set $V\subset \R$ and fix $w_{0}\in f^{-1}(V)$.   
Then for $N_{0}:=|w_{0}|+2$, one verifies that $B_{d}^{M_{0}}(w_{0}, \frac{1}{2^{N_{0}}})\subset f^{-1}(V)$, i.e.\ $f^{-1}(V)$ is open.  
Thus, $f:M_{0}\di \R$ is a continuous but unbounded function on a sequentially compact metric space $(M_{0}, d_{0})$, contradicting item \eqref{itemaa}.
Item \eqref{itemgas} also implies $\FC$ as $f$ is not uniformly continuous.  For item~\eqref{itemk}, $g:M_{0}\di \R$ defined as $g(w):=\frac{1}{2^{|w|}}$ is continuous in the 
same way as for $f$.  However, using $\IND_{2}$, the infimum of $g$ on $M_{0}$ is $0$, but there is no $w\in M_{0} $ with $g(w)=_{\R}0$, by definition.  Hence, item \eqref{itemk} also implies $\FC$.

\smallskip

Now assume item \eqref{diniitem} and suppose $\FC$ is again false; letting $(X_{n})_{n\in \N} $ and $( M_{0},d_{0})$ be as in the previous paragraph, we define $f_{n}: (\N\times M_{0})\di \R$ as:
\be\label{frion}
f_{n}(w):=
\begin{cases}
|w| & \textup{ if $|w|\leq n$}\\
0 & \textup{ otherwise}
\end{cases}.
\ee
Clearly, $\lim_{n\di \infty}f_{n}(w)=f(w)$ and $f_{n}(w)\leq f_{n+1}(w)$ for $w\in M_{0}$; $f_{n}$ is continuous in the same way as for $f$.  
Item \eqref{diniitem} implies that the convergence is \emph{uniform}, i.e.\
\be\label{prev}\textstyle
(\forall k\in \N)(\exists N\in \N)(\forall w\in M_{0})(\forall n\geq N)( |f_{n}(w)-f(w)|<\frac{1}{2^{k}}),
\ee
which yields a contradiction by letting $N_{1}\in \N$ be as in \eqref{prev} for $k=1$ and choosing $w_{1}\in M$ of length $N_{1}+1$ using $\IND_{2}$. 
One derives $\FC$ from item \eqref{equi} in the same way.   

\smallskip

Next, regarding item \eqref{cml}, suppose $\FC$ is false and consider again $(M_{0}, d_{0})$.
Define the continuous function $f:M_{0}\di \R$ by $f(w)=q_{|w|}$ where $(q_{n})_{n\in \N}$ is an enumeration of the rationals without repetitions.  
Using $\IND_{2}$, we have $f(M_{0})=\Q$ and the latter is not closed while $M_{0}$ is, contradicting item \eqref{cml}, and $\FC$ must hold.
To obtain the latter from item \eqref{discu}, note that $(M_{0}, d_{0})$ is infinite (using $\IND_{2}$) while 
all functions $f:M_{0}\di \R$ are continuous as \eqref{coki} is vacuously true.  
Regarding item \eqref{AA}, assuming again that $\FC$ is false, the sequence $(f_{n})_{n\in \N}$ as in \eqref{frion} is equicontinuous:
\[\textstyle
(\forall k\in \N, w\in M_{0})(\exists N\in \N)(\forall v\in B_{d_{0}}^{M_{0}}(w, \frac{1}{2^{N}}))(\forall n\in \N)(|f_{n}(w)-f_{n}(v)|<_{\R} \frac{1}{2^{k}}).
\]
which (vacuously) holds in the same way as for \eqref{coki}.  However, as for item \eqref{diniitem}, uniform convergence (of a sub-sequence) is false, i.e.\ item \eqref{AA} also implies $\FC$.
For item \eqref{wikked}, suppose $\FC$ is false, define $C_{n}:= \{w\in M_{0}: |w|>n+1 \} $, and verify that this closed and non-empty set has diameter at most $\frac{1}{2^{n}}$, using $\IND_{2}$.  
Since $\cap_{n\in \N}C_{n}=\emptyset$, we obtain $\FC$ from item  \eqref{wikked}.  For item \eqref{dizzy}, suppose $\FC$ is false, and define $C= \{ w\in M_{0}: |w| \textup{ is odd}\} $ and $D= \{ w\in M_{0}: |w| \textup{ is even}\} $.
One readily verifies that $C\cap D=\emptyset$, $C, D$ are closed, and $d(C, D)=0$.  

\smallskip

To establish the items in the theorem in $\ACAo+\QFAC^{0,1}$, the usual proof-by-contradiction goes through.  
A proof sketch of item \eqref{itemaa} as follows: let $(M, d)$ be as in the latter and suppose $f:M\di \R$ is continuous and unbounded, i.e.\ $(\forall n\in \N)(\exists x\in M)(|f(x)|>n)$.  
Since $M\subset \R$ and real numbers are represented by elements of Baire space, we may apply $\QFAC^{0,1}$ to obtain $(x_{n})_{n\in \N}$ in $M$ such that $|f(x_{n})|>n$ for all $n\in \N$.  
Since $M$ is sequentially compact, $(x_{n})_{n\in \N}$ has a convergent sub-sequence, say with limit $y\in M$.  Clearly, $f$ is not continuous at $y\in M$, a contradiction.  
To obtain \eqref{discu}, apply $\QFAC^{0,1}$ to the statement that $M\subset [0,1]$ is infinite, yielding a sequence $(w_{n})_{n\in \N}$ in $M$.  
Now define $f:M\di \R$ as $f(x_{n})=n$ and $f(y)=0$ for $y\ne x_{m}$ for all $m\in \N$.  Since $f$ is unbounded on $M$, it is discontinuous by item \eqref{itemaa}.  
Most other items are established using $\QFAC^{0,1}$ in the same way. 

\smallskip

We also sketch how $\QFAC^{0,1}$ implies item \eqref{cml}.  To this end, let $f, M$ be as in the closed map lemma and suppose $f(C)$ is not closed for closed $C\subset M$. 
Hence, there is $y_{0}\not \in f(C)$ such that $(\forall k\in \N)(\exists y\in B(y_{0}, \frac{1}{2^{k}}))( y\in f(C) )$.  
By definition, the latter formula means $(\forall k\in \N)(\exists x\in C)( |f(x)-y_{0}|<\frac{1}{2^{k}} )$.
Apply $\QFAC^{0,1}$ to obtain a sequence $(x_{n})_{n\in \N}$ in $C$ with $(\forall k\in \N)( |f(x_{k})-y_{0}|<\frac{1}{2^{k}} )$.
By sequential compactness, there is a convergent sub-sequence $(z_{n})_{n\in \N}$, say with limit $z$.  
Since $C$ is closed, we have $z\in C$ and since $f$ is continuous (and hence sequentially continuous) also $f(z)=y_{0}$.
This contradicts $y_{0}\not\in f(C)$ and the closed map lemma therefore follows from $\QFAC^{0,1}$.
\end{proof}
The final part of the proof also goes through if $f$ is only usco (see Def.\ \eqref{flung}). 
As to other generalisations of Theorem \ref{klonz}, the latter still goes through for `continuity' replaced by `absolute differentiability' from \cite{incell} formulated\footnote{The correct formulation based on \cite{incell} is that `$f:M\di \R$ is (absolutely) differentiable on the metric space $(M, d)$' in case we have
\[\textstyle
(\forall k\in \N, p\in M)(\exists N\in \N)(\forall x, y\in M)(0<d(x, p), d(y, p)<\frac{1}{2^{N}}\di \big| \frac{|f(x)-f(p)|}{d(x, p)}-  \frac{|f(y)-f(p)|}{d(y, p)} \big| <\frac{1}{2^{k}}),
\]
which is the `epsilon-delta' definition formulated to avoid the existence of the derivative. 
} appropriately.  

\smallskip

Finally, we observe that $(M_{0}, d_{0})$ from \eqref{MZ} is not (countably) compact, i.e.\ we need a slightly different approach for the latter, to be found in the next section.

\subsubsection{Compact spaces}\label{gawu2}
In this section, we establish that basic properties of (countably) compact spaces inhabit the range of hyperarithmetical analysis.

\smallskip

First of all, the following theorem is a version of Theorem \ref{klonz} for (countably) compact spaces and sequential continuity.  
We seem to (only) need sequential compactness to guarantee that everything remains provable in $\ACAo+\QFAC^{0,1}$.  
\begin{thm}[$\ACAo+\IND_{2}$]\label{klonz2}
The principle $\FC$ follows from any of the items \eqref{itemaaz}-\eqref{AAz} where $(M, d)$ is any metric space with $M\subset \R$; the principle $\QFAC^{0,1}$ implies all these items.
\begin{enumerate}
\renewcommand{\theenumi}{\alph{enumi}}
\item For \(weakly\) countably-compact and sequentially compact $(M, d)$, any sequentially continuous $f:M\di \R$ is bounded.\label{itemaaz}
\item The previous item with `is bounded' replaced by `is \(uniformly\) continuous'.\label{zama}
\item For a \(weakly\) countably-compact $(M, d)$ that is infinite, there is $f:M\di \R$ that is not sequentially continuous.\label{discu2}
\label{itemgasz}
\item The first item with `\(weakly\) countably-compact' replaced by `compact' or `complete and totally bounded'.\label{AAz}
\end{enumerate}
\end{thm}
\begin{proof}
We first derive $\FC$ from item \eqref{itemaaz} via a proof-by-contradiction.  
To this end, fix a sequence of non-empty finite sets of reals $(X_{n})_{n\in \N}$.  Suppose there is no sequence $(x_{n})_{n\in \N}$ of reals such that $x_{n}\in X_{n}$ for all $n\in \N$ and recall $M_{0}$ from \eqref{MZ}. 
Now define $M_{1}=M_{0}\cup \{0_{M_{1}}\}$ where $0_{M_{1}}$ is a new symbol such that $w\ne_{M_{1}} 0$ for $w\in M_{0}$ and `$=_{M_{1}}$' is `$=_{M_{0}}$' otherwise.  
Define $d_{1}:M_{1}^{2}\di \R$ as $d_{0}$ on $M_{0}$, as $d_{1}(w, 0_{M_{1}}):=d(0_{M_{1}}, w)=\frac{1}{2^{|w|}}$ for $w\in M_{0}$, and $d_{1}(0_{M_{1}}, 0_{M_{1}})=0$. 
Then $(M_{1}, d_{1})$ is a metric, which is shown in the same way as for $(M_{0}, d_{0})$.  

\smallskip

To show that $(M_{1}, d_{1})$ is countably-compact, let $(O_{n})_{n\in \N}$ be an open cover of $M_{1}$ and suppose $n_{1}\in \N$ is sucht hat $0_{M_{1}}\in O_{n_{1}}$. 
By definition, there is $N_{1}\in \N$ such that $B_{d_{1}}^{M_{1}}(0_{M_{1}}, \frac{1}{2^{N_{1}}})\subset O_{n_{0}}$, i.e.\ $d(0_{M_{1}},w )=\frac{1}{2^{|w|}}< \frac{1}{2^{N_{1}}}$ implies $w\in O_{n_{1}}$ for $w\in M_{0}$.  
Now use $\IND_{2}$ to enumerate the finitely many $v\in M_{0} $ such that $|v|\leq N_{1}$.  This finite sequence is covered by some $\cup_{n\leq n_{2}}O_{n}$, i.e.\ we have obtained a finite sub-covering of $M_{1}$, namely $\cup_{n\leq \max (n_{1},n_{2})}O_{n}$.   Moreover, $(M_{1}, d_{1})$ is sequentially compact, which can be proved via the same case distinction as for $(M_{0}, d_{0})$ in the proof of Theorem \ref{klonz}. 

\smallskip

Next, define $g:M_{1}\di \R$ as $g(w):=|w|$ for $w\in M_{0}$ and $g(0_{M_{1}})=_{\R}0$, which is clearly not bounded on $M_{1}$; this follows again via $\IND_{2}$.  
Then $g$ is continuous at $w_{0}\in M_{0}$ in the same way as $f$ from the proof of Theorem \ref{klonz}, namely since \eqref{coki} is vacuously true.   
To show that $g$ is sequentially continuous at $0_{M_{1}}$, let $(w_{n})_{n\in \N}$ be a sequence converging to $0_{M_{1}}$.  
In case this sequence is eventually constant $0_{M_{1}}$, clearly $g(0_{M})=\lim_{n\di \infty }g(w_{n})$ as required.  In case $(w_{n})_{n\in \N}$ is not eventually constant $0_{M_{1}}$, the convergence to $0_{M_{1}}$ in the $d_{1}$-metric
implies that for any $n\in \N$, there is $m\geq n$ with $|w_{m}|>n$.  Thus, $(w_{n})_{n\in \N}$ yields a sequence $(x_{n})_{n\in \N}$ such that $x_{n}\in X_{n}$ for all $n\in \N$, which contradicts our assumptions, i.e.\ this case cannot occur.  As a result, $g:M_{1}\di \R$ is sequentially continuous.  
Since, it is also unbounded (thanks to $\IND_{2}$), we obtain a contradiction with item \eqref{itemaaz}.
Thus, \eqref{itemaaz} implies $\FC$, and the same for item \eqref{zama}.  
To obtain $\FC$ from item~\eqref{discu2}, note that $(M_{1}, d_{1})$ is infinite (using $\IND_{2}$) while 
\emph{all} functions $f:M_{1}\di \R$ are \emph{sequentially} continuous by the previous.  

\smallskip

To show that $(M_{1}, d_{1})$ satisfies the properties in item \eqref{AAz}, note that for $\Psi:M_{1}\di \R^{+}$, the ball $B_{d_{1}}^{M_{1}}(0_{M_{1}},\Psi(0_{M_{1}})  )$
covers all but finitely many points of $M_{1}$ (in the same way as $O_{n_{0}}$ from the second paragraph of the proof).  Hence, $(M_{1}, d_{1})$ is compact, and totally boundedness follows in exactly the same way. 
For completeness, let $(w_{n})_{n\in \N}$ be a Cauchy sequence in $M_{1}$, i.e.\ we have
\[\textstyle
(\forall k\in \N)(\exists N\in \N)(\forall n, m\geq N)( d_{1}(w_{n}, w_{m})<\frac{1}{2^{k}}).  
\]
As above, $(w_{n})_{n\in \N}$ is either eventually constant or provides a sequence $(x_{n})_{n\in \N}$ such that $x_{n}\in X_{n}$ for all $n\in \N$.  The latter case is impossible by assumption and the former case is trivial.   

\smallskip

To establish the items in the theorem in $\ACAo+\QFAC^{0,1}$, the usual proof-by-contradiction goes through as in the proof of Theorem \ref{klonz}.
\end{proof}
We believe that we cannot use epsilon-delta or topological continuity in the previous theorem.  
Nonetheless, we have the following corollary that makes use of the sequential\footnote{A function $f:M\di \R$ is called \emph{sequentially uniformly continuous} if for any sequences $(w_{n})_{n\in \N}$, $(v_{n})_{m\in \N}$ in $M$ such that $\lim_{n\di \infty}d(w_{n}, v_{n})=0$, we have $\lim_{n\di \infty}|f(w_{n})-f(v_{n})|=0$.} definition of uniform continuity.  
\begin{cor}
Over $\ACAo+\IND_{2}$, items \eqref{itemaa}-\eqref{discu} in Theorem \ref{klonz} and items \eqref{itemaaz}-\eqref{AAz} in Theorem \ref{klonz2} are intermediate between $\QFAC^{0,1}$ and $\FC$ if we replace `continuity' by `sequential uniform continuity'.
\end{cor}
\begin{proof}
The usual proof-by-contradiction using $\QFAC^{0,1}$ (and $(\exists^{2})$) shows that sequential uniform continuity implies uniform continuity.  
For the remaining implications, consider $g:M_{1}\di \R$ from the the proof of Theorem \ref{klonz2}.  This function is sequentially continuous at $0_{M_{1}}$ since any sequence converging to $0_{M_{1}}$ must be eventually constant $0_{M_{1}}$.  
Similarly, for sequences $(w_{n})_{n\in \N}$, $(v_{n})_{n\in \N}$ in $M_{1}$ $\lim_{n\di \infty}d_{1}(w_{n}, v_{n})=0$ implies that the sequences are eventually equal.    
Hence, $g$ is also sequentially uniformly continuous.   A similar proof goes through for $(M_{0}, d_{0})$ and $f$ from Theorem \ref{klonz}.
\end{proof}
We have identified a number of basic properties of continuous functions on compact metric spaces that exist in the range of hyperarithmetical analysis.  
A number of restrictions and variations are possible, which is the topic of the next section.

\subsubsection{Restrictions}\label{gawu3}
We show that some the above principles still inhabit the range of hyperarithmetical analysis if we restrict to arithmetically defined objects or Lipschitz continuity.

\smallskip

First of all, the following theorem establishes that Theorem \ref{klonz} holds if we restrict to arithmetically defined objects.
\begin{thm}\label{slonk}
Item \eqref{itemaa} from Theorem \ref{klonz} still implies $\WSAC$ over $\ACAo+\IND_{1}$ if all objects in the former item have an arithmetical definition. 
\end{thm}
\begin{proof}
In a nutshell, we can modify the above proofs to obtain (only) $\WSAC$ while all relevant objects can be defined using $(\exists^{2})$.  
To this end, let $\varphi$ be arithmetical and such that $(\forall n\in \N)(\exists! X\subset \N)\varphi(n, X)$, but there is no sequence $(X_{n})_{n\in \N}$ with $(\forall n\in \N)\varphi(n, X_{n})$.  
Use $\exists^{2}$ to define $\eta:[0,1]\di (2^{\N}\times 2^{\N})$ such that $\eta(x)=(f, g)$ outputs the binary expansions of $x$, taking $f=g$ if there is only one.  
Define the following set using $\exists^{2}$:
\[
A:=\{x\in [0,1]:  (\exists m\in \N)(\varphi(m,\eta(x)(1))\vee \varphi(m, \eta(x)(2))) \}
\]
and $Y(x):=(\mu m)[\varphi(m,\eta(x)(1)\vee \varphi(m, \eta(x)(2)))]$.
Then $Y$ is injective and surjective on $A$. In particular $A_{m}:=\{x\in [0,1]: \varphi(m,\eta(x)(1))\vee \varphi(m, \eta(x)(2)) \} $ contains exactly one element by definition. 
Using $X_{n}=A_{n}$, the metric space $(M_{0}, d_{0})$ as in \eqref{MZ} in the proof of Theorem \ref{klonz} now has an arithmetical definition.  The same holds for the function $F:M_{0}\di \R$ where $F(w):=|w|$.  The rest of the proof of item~\eqref{itemaa} of Theorem \ref{klonz} now goes through, using $\IND_{1}$ instead of $\IND_{2}$ where relevant, yielding in particular a contradiction.  Hence, there must be a sequence $(X_{n})_{n\in \N}$ with $(\forall n\in \N)\varphi(n, X_{n})$, i.e.\ $\WSAC$ follows as required.
\end{proof}
Secondly, we show that we may replace `continuity' by `Lipschitz continuity' in some of the above principles.  
\bdefi A function $f:M\di \R$ is \emph{$\alpha$-H\"older-continuous} in case there exist $M, \alpha>0$ such that for any $v, w\in M$:
\[
|f(v)-f(w)|\leq M d(v, w)^{\alpha},
\]
A function is \emph{Lipschitz \(continuous\)} if is is $1$-H\"older-continuous.  
\edefi
\begin{thm}[$\ACAo+\IND_{2}$]\label{klonzzz}
The principle $\FC$ follows from any of the items \eqref{prepo}-\eqref{itemkkk} where $(M, d)$ is any metric space with $M\subset \R$; the principle $\QFAC^{0,1}$ implies items \eqref{prepo}-\eqref{itemkkk}.
\begin{enumerate}
\renewcommand{\theenumi}{\alph{enumi}}
\item For a metric space $(M, d)$, any sequentially compact $C\subseteq M$ is bounded, i.e.\ there are $w\in M, m\in \N$ with $(\forall v\in C)(d(v, w)\leq m)$ \(see \cite{bartle2}*{p.\ 333}\).\label{prepo}
\item For sequentially compact $(M, d)$, any uniformly continuous $f:M\di \R$ is bounded.\label{itemak}
\item The previous item with `uniformly' replaced by `$\alpha$-H\"older' or `Lipschitz'.  
\label{itemgask}
\item For sequentially compact $(M, d)$ that is infinite, there exists $f:M\di \R$ that is bounded but not Lipschitz continuous.\label{discu3}
\item For sequentially compact and bounded $(M, d)$ and Lipschitz $f:M\di \R$ with $\inf_{x\in M}f(x)=y\in \R$ given, there is $x\in M$ with $f(x)=y$.\label{itemkkk}
\end{enumerate}
\end{thm}
\begin{proof}
First of all, to derive item \eqref{prepo} from $\QFAC^{0,1}$, fix a metric space $(M, d)$ and let $C\subseteq M$ be sequentially compact.  
Suppose $C$ is not bounded, i.e.\ for some fixed $w_{0}\in M$, we have $(\forall m\in \N)(\exists v\in C)(d(w_{0}, v)>m )$.  Apply $\QFAC^{0,1}$ to obtain a sequence $(v_{n})_{n\in \N}$ 
such that $d(w_{0}, v_{n})>n$ for all $n\in \N$.  Clearly, this sequence cannot have a convergent sub-sequence, a contradiction, and $C$ must be bounded. 
To derive item \eqref{discu3} from $\QFAC^{0,1}$, apply $\QFAC^{0,1}$ to the statement that $M$ is infinite.  
The resulting sequence $(w_{n})_{n\in \N}$ has a convergent sub-sequence, say $(v_{n})_{n\in \N}$ with limit $v$.  
Define $f(w)= 1$ (resp.\ $f(w)=-1$) if $v_{n}=w$ and $n$ is even (resp.\ odd), and $f(w)=0$ otherwise.  
Clearly, $f:M\di \R$ is bounded but not (Lipschitz) continuous.  
By, Theorem \ref{klonz}, the other items follow from $\QFAC^{0,1}$.  

\smallskip

Secondly, to derive $\FC$ from item \eqref{prepo}, suppose $(X_{n})_{n\in \N}$ is a sequence of finite sets such that there is no sequence $(x_{n})_{n\in \N}$ with $x_{n}\in X_{n}$ for all $n\in \N$.  
Recall the set $M_{0}$ from \eqref{MZ} and define $d_{2}:M_{0}^{2}\di \R$ as $d_{2}(v, w)=\big| |v|-|w| \big|$ for $v, w\in M_{0}$.
That $d_{2}$ is a metric is readily verified: the first and third item of Definition \ref{donkc} hold by definition and the triangle equality of the absolute value; the second item in this definition holds 
since $d_{2}(v, w)=0\asa |u|=|w|\asa u=_{M_{0}} w$.  
Now, the set $C=\{w\in M_{0}:  |w| \textup{ is even}\}$ is sequentially compact, as every sequence in $C$ either has at most finitely many different members, or  yields a sequence $(x_{n})_{n\in \N}$ such that $x_{n}\in X_{n}$ for all $n\in \N$.  
We have excluded the latter by assumption, while the former trivially yields a convergent sub-sequence.  
Using $\IND_{2}$, $C$ is however not bounded in $(M_{0}, d_{2})$,  a contradiction, and item \eqref{prepo} implies $\FC$. 

\smallskip

Thirdly, to derive $\FC$ from the remaining items, let $(M_{0}, d_{2})$ be as above and note that the latter is sequentially compact as in the previous paragraph.  Now define $f:M_{0}\di \R$ as $f(u):= \frac{|u|}{2}$ and observe that $|f(u)-f(v)|= \frac{1}{2}\big| |u|-|v|\big|\leq \frac{1}{2}d_{2}(u, v)$, i.e.\ $f$ is Lifschitz (and uniformly) continuous.  
However, $\IND_{2}$ shows that $f$ is not bounded, a contradiction, and items \eqref{itemak}-\eqref{itemgask} imply $\FC$. 
Similarly, item \eqref{discu3} implies $\FC$ as $(M_{0}, d_{2})$ is such that every bounded function $f:M_{0}\di \R$ is automatically Lipschitz.
Indeed, if $|f(w)|\leq M_{0}$ for all $w\in M_{0}$, then the Lipschitz constant for $f$ can be taken to be $2M_{0}$. 

\smallskip

Finally, to derive $\FC$ from item \eqref{itemkkk}, suppose the former is false and consider again $(M_{0}, d_{0})$, which is trivially bounded due to the definition of $d_{0}$.  
Now define $g:M_{0}\di \R$ as $g(u):=\frac{1}{2^{|u|+1}}$.  This function is Lipschitz on $(M_{0}, d_{0})$ as  
\[\textstyle
|g(u)-g(v)|=|\frac{1}{2^{|u|+1}}-\frac{1}{2^{|v|+1}}|=\frac{1}{2}|\frac{1}{2^{|u|}}-\frac{1}{2^{|v|}}|=\frac{1}{2}d(u, v).
\]
However, $g$ has infimum equal to zero (using $\IND_{2}$) but is strictly positive on $M_{0}$, contradicting item \eqref{itemkkk}, which establishes the theorem. 
\end{proof}
In conclusion, many implications between the notions in Definition \ref{deaco} exist in the range of hyperarithmetical analysis, as well 
as the associated Lebesgue number lemma for countable coverings of open sets.  These are left to the reader.  

\subsubsection{Connectedness}\label{gawu4}
We show that basic properties of connected metric spaces exist in the range of hyperarithmetical analysis, including the intermediate value theorem. 
We also obtain some elegant equivalences in Theorem \ref{preppiu}. 

\smallskip

First of all, Cantor and Jordan were the first to study connectedness (\cite{wilders}), namely as in the first item in Definition \ref{chainc}.  
The connectedness notions from the latter are equivalent for compact metric spaces in light of \cite{mannetti}*{\S4.39} or \cite{pugh}*{p.\ 123}.
\bdefi[Connectedness]\label{chainc}~
\begin{itemize}
\item A metric space $(M, d)$ is \emph{chain connected} in case for any $w, v\in M$ and $\eps>0$, there is a sequence $w=x_{0},x_{1}, \dots, x_{n-1},{x_{n}}=v\in M$ such that for all $i< n$ we have $d(x_{i}, x_{i+1})<\eps$. 
\item A metric space $(M, d)$ is \emph{connected} in case $M$ is not the disjoint union of two non-empty open sets.  
\end{itemize}
\edefi
\noindent
We shall study the following generalisation of the intermediate value theorem.
\begin{princ}[{Intermediate Value Theorem}]\label{IVT}
Let $(M, d)$ be a sequentially compact and chain connected metric space and let $f:M\di \R$ be continuous.
If $f(w)<c<f(v)$ for some $v, w\in M$ and $c\in \R$, then there is $u\in M$ with $f(u)=c$.  
\end{princ}
The \emph{approximate} intermediate value theorem is the previous principle with the conclusion weakened to `then for any $\eps>0$ there is $u\in M$ with $|f(u)-c|<\eps$.'
The latter theorem is well-known from constructive mathematics (see e.g.\ \cite{bridge1}*{p.\ 40}).  
\begin{thm}[$\ACAo+\IND_{2}$] \label{preppiu}
The principle $\FC$ follows from any of the items \eqref{ivt1}-\eqref{ivt2} where $(M, d)$ is any metric space with $M\subset \R$; the principle $\QFAC^{0,1}$ implies items \eqref{ivt1}-\eqref{ivt2}.
\begin{enumerate}
\renewcommand{\theenumi}{\alph{enumi}}
\item The intermediate value theorem as in Principle \ref{IVT}.  \label{ivt1}
\item Principle \ref{IVT} for Lipschitz continuous functions.
\item The approximate intermediate value theorem.\label{ivt25}
\item Let $(M, d)$ be a sequentially compact and chain connected metric space and let $f:M\di \{0, 1\}$ be continuous.  Then $f$ is constant on $M$.  \label{ivt3}
\item Let $(M, d)$ be a sequentially compact and chain connected metric space and let $f:M\di \R$ be locally constant.  Then $f$ is constant on $M$.  \label{ivt4}
\item Let $(M, d)$ be sequentially compact and chain connected and let $f:M\di \R$ be locally constant and continuous.  Then $f$ is constant on $M$.\label{ivt5}
\item For a sequentially compact metric space $(M, d)$, chain connectedness implies connectedness \label{ivt2}
\item Let $(M, d)$ be a sequentially compact and chain connected metric space and let $f:M\di \R$ be \(Lipschitz\) continuous.  Then $f$ is bounded on $M$.  \label{h1}
\item Item \eqref{h1} with `$f$ is bounded' replaced by `$f(M)$ is not dense in $\R$'.\label{h2}
\item Item \eqref{h1} with `$f$ is bounded' replaced by `$f(M)$ is closed'.\label{h3}
\end{enumerate}
Moreover, items \eqref{ivt1}, \eqref{ivt25}, and \eqref{ivt3}-\eqref{ivt2} are equivalent. 
\end{thm}
\begin{proof}
First of all, we show that item \eqref{ivt1} implies $\FC$.  To this end, suppose the latter is false and consider $M_{0}$ as in \eqref{MZ}.  Let $(q_{n})_{n\in \N}$ be an enumeration of the rationals (without repetitions) and define $d_{3}:M_{0}^{2}\di \R$ as follows: $d_{3}(w, v):= |q_{|w|}-q_{|v|}|$ for $w, v\in M_{0}$.
Then $(M_{0}, d_{3})$ is a sequentially compact metric space, which is proved in the same way as for the previous metrics $d_{0}, d_{1}, d_{2}$, namely that any sequence in $M_{0}$ can have at most finitely many different elements.  
That $(M_{0}, d_{3})$ is chain connected is proved using $\IND_{2}$.   Indeed, fix $u, w\in M_{0}, \eps>0$ and consider $d_{3}(w, v)=|q_{|w|}-q_{|v| }|$.  
Let $q_{|w|}=r_{0},r_{1}, \dots, r_{k-1}, r_{k}=q_{|v|}\in \Q$ be a finite sequence such that $|r_{i}-r_{i+1}|<\eps$ for $i<k$.  Using $\IND_{2}$, there are $w_{i}\in M_{0}$ such that $q_{|w_{i}|}=r_{i}$ for $i<k$, and chain connectedness of $M_{0}$ follows. 

\smallskip
 
Now define $f:M_{0}\di \R$ by $f(w)=\frac{1}{2}q_{|w|}$, which is (Lipschitz) continuous, essentially by the definition of $d_{3}$, as we have:
\[\textstyle
|f(w)-f(v)|=|\frac{1}{2}q_{|w|}-\frac{1}{2}q_{|v|}|=\frac{1}{2}|q_{|w|}-q_{|v|}|\leq \frac{1}{2} d_{3}(w, v).
\]
However, the range of $f$ consists of rationals, i.e.\ it does not have the intermediate value property.  
This contradiction yields $\FC$.  The same proof goes through for items \eqref{h1}-\eqref{h3}. 

\smallskip

Secondly, assume $\QFAC^{0,1}$ and let $(M, d), f:M\di \R, w, v\in M,$ and $ c\in \R$ be as in Principle \ref{IVT}.  Since $M$ is chain connected, we have 
\be\label{tach}\textstyle
(\forall k\in \N)(\exists z^{1^{*}})( z(0)=w\wedge z(|w|-1)=v \wedge (\forall i<|z|-1) d(z(i), z(i+1))<\frac{1}{2^{k}}    ).
\ee
Apply $\QFAC^{0,1}$ to obtain a sequence $(z_{k})_{k\in \N}$ of finite sequences. 
Define a sequence $(t_{k})_{k\in \N}$ in $M$ where $t_{k}$ is the first element $t$ in $z_{k}$ such that $f(t)\geq c$.  
By sequential completeness, there is a convergent sub-sequence $(s_{k})_{k\in \N}$ with limit $s\in M$. 
Since $f$ is continuous, we have $\lim_{k\di \infty}f(s_{k})=f(s)$ and hence $f(s)=c$. 

\smallskip

Thirdly, to show that item \eqref{ivt2} implies $\FC$, again suppose the latter is false and consider $(M_{0}, d_{3})$.   
By the above, the latter is sequentially compact and chain connected.  To show that it is not connected, define 
$O_{1}:=\{w\in M_{0}: q_{|w|}>\pi\}$ and $O_{2}:=\{w\in M_{0}: q_{|w|}<\pi\}$, verify that they are open and disjoint, and observe that $M_{0}=O_{1}\cup O_{2}$, i.e.\ item \eqref{ivt2} is false. 
Note also that $f:M_{0}\di \R$ defined as $f(w)=1$ if $w\in O_{1}$ and $0$ otherwise, is continuous but not constant, i.e.\ item \eqref{ivt3} also implies $\FC$.

\smallskip

To derive item \eqref{ivt2} from $\QFAC^{0,1}$, let $(M, d)$ be as in the former, i.e.\ sequentially compact and chain connected.  
Suppose $M$ is not connected, i.e.\ $M=O_{1}\cup O_{2}$ where the latter are open, disjoint, and non-empty.
Now fix $v\in O_{1}$ and $w\in O_{2}$ and consider \eqref{tach}.  Apply $\QFAC^{0,1}$ to obtain a sequence $(z_{k})_{k\in \N}$ of finite sequences. 
Define sequences $(s_{k})_{k\in \N}$ and $(t_{k})_{k\in \N}$ in $M$ where $t_{k}$ is the first element $t$ in $z_{k}$ such that $t\in O_{2}$ and $s_{k}$ is the predecessor of $t$ in $z_{k}$.    
By sequential completeness, $(s_{k})_{k\in \N}$ and $(t_{k})_{k\in \N}$ have convergent sub-sequences, with the same limit by construction.  
However, if this limit is in $O_{1}$, then so is $(t_{k})_{k\in \N}$ eventually, a contradiction.  Similarly, if this limit is in $O_{2}$, then so is $(s_{k})_{k\in \N}$ eventually, a contradiction.
In each case we obtain a contradiction, i.e.\ $M$ must be connected, and item \eqref{ivt2} follows.   The same proof goes through for item \eqref{ivt3}.

\smallskip

Next, item \eqref{ivt2} implies item \eqref{ivt4} as in case the latter fails for $f:M\di \R$, say with $f(w)<_{\R}f(v)$, then $O_{1}=\{z\in M: f(z)\leq f(w)\}$ and $O_{2}=\{ z\in M: f(z)>f(w) \} $ are open, disjoint, and non-empty sets such that $M=O_{1}\cup O_{2}$, i.e.\ item \eqref{ivt2} fails too.  To show that item \eqref{ivt4} implies $\FC$, suppose the latter is false and let $(M_{0}, d_{3})$ be as above. Define $g:M_{0}\di \R$ as $g(w)=n$ in case $|q_{|w|}|\in [n\pi, (n+1)\pi]$.  
Clearly, $f$ is locally constant but not constant, i.e.\ item~\eqref{ivt4} is false.  
To derive item~\eqref{ivt2} from item~\eqref{ivt4} (and item \eqref{ivt5}), suppose the former is false, i.e.\ $(M, d)$ is a sequentially compact and chain connected metric space that is not connected.
Let $M=O_{1}\cup O_{2}$ be the associated decomposition and note that $f:M\di\R$ defined by $f(w)=1$ if $w\in O_{1}$ and $0$ otherwise, is locally constant (and continuous) but not constant, i.e.\ item \eqref{ivt4} (and \eqref{ivt5}) also fails. 
The equivalence for item \eqref{ivt3} follows in the same way.  

\smallskip

To show that item \eqref{ivt2} implies item \eqref{ivt1}, suppose the latter is false for $f:M\di \R$ and $c\in \R$, i.e.\ $f(w)\ne c$ for all $w\in M$.  
By assumption, $O_{1}:= \{w\in M: f(w)<c\}$ and $O_{2}:= \{w\in M:f(w)>c\}$ are open, disjoint, and non-empty, i.e.\ item~\eqref{ivt2} also fails. 
To show that item \eqref{ivt1} (and item \eqref{ivt25}) implies item \eqref{ivt2}, suppose the latter fails for $M=O_{1}\cup O_{2}$, i.e.\ the latter are open, non-empty, and disjoint. 
Then $f:M\di \R$ defined by $f(w)=1$ if $w\in O_{1}$ and $0$ otherwise, is continuous but does not have the (approximate) intermediate value property.  
\end{proof}
Regarding item \eqref{h2}, we could not find a way of replacing `$f(M)$ is not dense in $\R$' by `$f(M)$ has finite measure'.   
We could study \emph{local connectedness} and obtain similar results, but feel this section is long enough as is.  

\smallskip

In conclusion, we have identified many basic properties of metric spaces that exist in the range of hyperarithmetical analysis. 
We believe there to be many more such principles in e.g.\ topology.  

\section{Functions of Bounded variation and around}\label{BVS}
We introduce functions of bounded variation (Section \ref{deffer}) and show that their basic properties exist in the range of hyperarithmetical analysis (Section~\ref{jefdoetbef}).  
Similar to Theorem \ref{slonk}, we could restrict to arithmetically defined functions.  
\subsection{Bounded variation and variations}\label{deffer}
The notion of \emph{bounded variation} (often abbreviated $BV$) was first explicitly\footnote{Lakatos in \cite{laktose}*{p.\ 148} claims that Jordan did not invent or introduce the notion of bounded variation in \cite{jordel}, but rather discovered it in Dirichlet's 1829 paper \cite{didi3}.} introduced by Jordan around 1881 (\cite{jordel}) yielding a generalisation of Dirichlet's convergence theorems for Fourier series.  
Indeed, Dirichlet's convergence results are restricted to functions that are continuous except at a finite number of points, while $BV$-functions can have infinitely many points of discontinuity, as already studied by Jordan, namely in \cite{jordel}*{p.\ 230}.
In this context, the \emph{total variation} $V_{a}^{b}(f)$ of $f:[a, b]\di \R$ is defined as:
\be\label{tomb}\textstyle
\sup_{a\leq x_{0}< \dots< x_{n}\leq b}\sum_{i=0}^{n} |f(x_{i})-f(x_{i+1})|.
\ee
The following definition provides two ways of defining `$BV$-function'.  We have mostly studied the first one (\cites{dagsamXI, samBIG, samBIG3}) but will use the second one in this paper. 
\bdefi[Variations on variation]\label{varvar}~
\begin{enumerate}  
\renewcommand{\theenumi}{\alph{enumi}}
\item The function $f:[a,b]\di \R$ \emph{has bounded variation} on $[a,b]$ if there is $k_{0}\in \N$ such that $k_{0}\geq \sum_{i=0}^{n} |f(x_{i})-f(x_{i+1})|$ 
for any partition $x_{0}=a <x_{1}< \dots< x_{n-1}<x_{n}=b  $.\label{donp}
\item The function $f:[a,b]\di \R$ \emph{has total variation} $z\in \R$ on $[a,b]$ if $V_{a}^{b}(f)=z$.\label{donp2}
\end{enumerate}
\edefi
We recall the `virtual' or `comparative' meaning of suprema in RM from e.g.\ \cite{simpson2}*{X.1}.  In particular, a formula `$\sup A\leq b$' is merely 
shorthand for (essentially) the well-known definition of the supremum.  

\smallskip

Secondly, the fundamental theorem about $BV$-functions is formulated as follows.
\begin{thm}[Jordan decomposition theorem, \cite{jordel}*{p.\ 229}]\label{drd}
A $BV$-function $f : [0, 1] \di \R$ is the difference of  two non-decreasing functions $g, h:[0,1]\di \R$.
\end{thm}
Theorem \ref{drd} has been studied via second-order representations in \cites{groeneberg, kreupel, nieyo, verzengend}.
The same holds for constructive analysis by \cites{briva, varijo,brima, baathetniet}, involving different (but related) constructive enrichments.  
We have obtained many equivalences for the Jordan decomposition theorem, formulated using item \eqref{donp} from Definition \ref{varvar} in \cites{dagsamXI, samBIG3}, involving the following principle.
\begin{princ}[$\cocode_{0}$]
A countable set $A\subset [0,1]$ can be enumerated.
\end{princ}
This principle is `explosive' in that $\ACAo+\cocode_{0}$ proves $\ATR_{0}$ while $\FIVE^{\omega}+\cocode_{0}$ proves $\SIX$ (see \cite{dagsamX}*{\S4}).  

\smallskip

Thirdly, $f:\R\di \R$ is \emph{regulated} if for every $x_{0}$ in the domain, the `left' and `right' limit $f(x_{0}-)=\lim_{x\di x_{0}-}f(x)$ and $f(x_{0}+)=\lim_{x\di x_{0}+}f(x)$ exist.  
Feferman's $\mu$ readily provides the limit of $(f(x+\frac{1}{2^{n}}))_{n\in \N}$ if it exists, i.e.\ the notation $\lambda x. f(x+)$ for regulated $f$ makes sense in $\ACAo$.  
On a historical note, Scheeffer and Darboux study discontinuous regulated functions in \cite{scheeffer, darb} without using the term `regulated', while Bourbaki develops Riemann integration based on regulated functions in \cite{boerbakies}.  
Finally, $BV$-functions are regulated while Weierstrass' `monster' function is a natural example of a regulated function not in $BV$.

\subsection{Bounded variation and hyperarithmetical analysis}\label{jefdoetbef}
We identify a number of statements about $BV$-functions that exist within the range of hyperarithmetical analysis, assuming $\ACAo$.
We even obtain some elegant equivalences and discus the (plentiful) variations of these results in Section \ref{remvar}.  

\smallskip

First of all, the following principle appears to be important, which is just $\cocode_{0}$ from the previous section restricted to strongly countable sets. 
\begin{princ}[$\cocode_{1}$]
A strongly countable set $A\subset [0,1]$ can be enumerated.
\end{princ}
\noindent
Some RM-results for $\cocode_{1}$ may be found in \cite{dagsamXI}*{\S2.2.1}; many variations are possible and these systems all exist in the range of hyperarithmetical analysis. 
The cited results are not that satisfying as they mostly deal with properties of strongly countable sets, in contrast to the below.    

\smallskip

Secondly, we have the following theorem, establishing that items \eqref{fling}-\eqref{dagg} exist in the range of hyperarithmetical analysis. 
\begin{thm}[$\ACAo+\IND_{1}$]\label{xion}
The higher items imply the lower ones.  
\begin{enumerate}
\renewcommand{\theenumi}{\roman{enumi}}
\item The principle $\QFAC^{0,1}$.
\item \(Jordan\) For $f\in BV$ with $V_{0}^{1}(f)=1$, there are non-decreasing $g, h:[0,1]\di \R$ such that $f=g-h$.\label{fling}
\item For $f\in BV$ with $V_{0}^{1}(f)=1$, there is a sequence that includes all points of discontinuity of $f$.  \label{fling2}
\item For $f\in BV$ with $V_{0}^{1}(f)=1$, the supremum\footnote{To be absolutely clear, we assume, for the existence of a functional $\Phi:\Q^{2}\di\R$ such that $(\forall p, q\in \Q\cap [0,1])(\Phi(p, q)=\sup_{x\in [p,q ]}f(x))$). } $\sup_{x\in [p, q]}f(x)$ exists for $p, q\in [0,1]\cap\Q$.  \label{fling3}
\item $\cocode_{1}$.\label{dagg}
\item $\USAC$.
\end{enumerate}
Items \eqref{fling}-\eqref{fling2} are equivalent; we only use $\IND_{1}$ to derive $\cocode_{1}$ from item \eqref{fling3}.
\end{thm}
\begin{proof}
Assume $\QFAC^{0,1}$ and let $f\in BV$ be such that $V_{0}^{1}(f)=1$.  
By \cite{dagsamXIV}*{Theorem~2.16}, $\ACAo$ suffices to enumerate all jump discontinuities of a regulated function, while $f$ is regulated by \cite{dagsamXI}*{Theorem 3.33}.
Then $V_{0}^{1}(f)=1$ implies that
\[\textstyle
(\forall k\in \N)(\exists x_{0}, \dots, x_{m}\in I)\big[(\forall i<m)(x_{i}<x_{i+1})\wedge 1-\frac{1}{2^{k}}< \sum_{j=0}^{m}|f(x_{j})-f(x_{j+1})|\big].
\]
The formula in square brackets is arithmetical, i.e.\ since $(\exists^{2})$ is available we may apply $\QFAC^{0,1}$ to obtain a sequence of finite sequences $(w_{n})_{n\in \N}$ witnessing the previous 
centred formula.  This sequence includes all removable discontinuities of $f$.  Indeed, suppose $y_{0}\in [0,1]$ is such that $f(y_{0}-)=f(y_{0}+)\ne f(y_{0})$ is not among the reals in $(w_{n})_{n\in \N}$.  Let $k_{0}\in \N$ be such that $|f(y_{0}+)-f(y_{0})|>\frac{1}{2^{k_{0}}}$ and note that $1-\frac{1}{2^{k_{0}+1}}< \sum_{j=0}^{m_{k_{0}+1}}|f(x_{j})-f(x_{j+1})|$ for $w_{k_{0}+1}=(x_{0}, \dots, x_{m_{k_{0}+1}})$ by assumption. Extending $w_{k_{0}+1}$ with $y_{0}$ and points $z_{0}<y_{0}<u_{0}$ close enough to $y_{0}$, we obtain a partition of $[0,1]$ that witnesses that $V_{0}^{1}(f)>1$, contradicting our assumptions.  Since $f$ is regulated, it only has removable and jump discontinuities, i.e.\ item \eqref{fling2} follows from $\QFAC^{0,1}$ as required.  

\smallskip

By \cite{dagsamXI}*{Theorem 3.33}, $\ACAo$ suffices to enumerate the points of discontinuity of any monotone $g:[0,1]\di \R$, i.e.\ item \eqref{fling} implies item \eqref{fling2}.
To obtain item \eqref{fling} from item \eqref{fling2}, note that the supremum over $\R$ in \eqref{tomb} can be replaced by a supremum over $\Q$ and any sequence that includes all points of discontinuity of $f$.  
Hence, we may use $(\exists^{2})$ to define the weakly increasing function $g(x):= \lambda x.V_{0}^{x}(f) $.  One readily verifies that $h(x):= g(x)-f(x)$ is also weakly increasing, i.e.\ $f=g-h$ as in item \eqref{fling} follows.  To obtain item \eqref{fling3} from item \eqref{fling2}, note that -similar to the previous- the supremum over $\R$ in $\sup_{x\in [p, q]}f(x)$ can be replaced by a supremum over $\Q$ and any sequence that includes all points of discontinuity of $f$.  

\smallskip

To derive $\cocode_{1}$ from item \eqref{fling3}, let $A\subset [0,1]$ and $Y:[0,1]\di \R$ such that the latter is injective and surjective on the former.  
Now define $f:[0,1]\di \R$ as follows: $f(x):= \frac{1}{2^{Y(x)}}$ if $x\in A$, and $0$ otherwise.  Using $\IND_{1}$, $f$ is in $BV$ and $V_{0}^{1}(f)=1$.  
Now use $(\exists^{2})$ to decide whether $\sup_{x\in [0, \frac{1}{2}]}f(x)< 1$; if the latter holds, `$1$' is the first bit of the binary expansion of $x_{0}\in [0,1]$ such that $Y(x_{0})=0$.  
Using the supremum functional and $(\exists^{2})$, the usual interval-halving technique then allows us to enumerate $A$, as required for $\cocode_{1}$. 
For the final part, let $\varphi$ be arithmetical and such that $(\forall n\in \N)(\exists! X\subset \N )\varphi(n, X)$. 
Use $\exists^{2}$ to define $\eta:[0,1]\di (2^{\N}\times 2^{\N})$ such that $\eta(x)=(f, g)$ outputs the binary expansions of $x$, taking $f=g$ if there is only one.  
Then $E_{n}=\{x\in [0,1]:\varphi(n,\eta(x)(1))\vee \varphi(n, \eta(x)(2))\}$ is a singleton and $Y(x):=(\mu n)(x\in E_{n})$ is injective and surjective on $A=\cup_{n\in \N}E_{n}$.  
The enumeration of $A$ provided by $\cocode_{1}$ yields the consequent of $\USAC$. 
\end{proof}
As to the role of the Axiom of Choice in Theorem \ref{xion}, we note that the items \eqref{fling}-\eqref{dagg} can also be proved \emph{without} $\QFAC^{0,1}$.
Indeed, $\lambda x.V_{0}^{x}(f)$ as in \eqref{tomb} involves a supremum over $\R$, which can be defined in $\Z_{2}^{\Omega}$ using the well-known interval-halving technique, i.e.\ the 
usual textbook proof (see e.g.\ \cite{voordedorst}) goes through in $\Z_{2}^{\Omega}$.  

\smallskip

Thirdly, we have the following corollary using slightly more induction.
\begin{cor}~
Over $\ACAo+\IND_{2}$, item \eqref{fling2} from Theorem \ref{xion} is equivalent to:
\begin{center}
 for $f\in BV$ with $V_{0}^{1}(f)=1$ and with Fourier coefficients given, there is a sequence $(x_{n})_{n\in \N}$ outside of which the Fourier series converges to $f(x)$.  \label{fling4}
\end{center} 
\end{cor}
\begin{proof}
We note that $\IND_{2}$ suffices to guarantee that $BV$-functions are regulated by \cite{dagsamXI}*{Theorem 3.33}.
Now, the Fourier series of a $BV$-function always converges to $\frac{f(x+)+f(x-)}{2}$ \emph{and} this fact is provable in $\ACAo$ if the Fourier coefficients are given, as discussed in (a lot of detail in) \cite{samBIG}*{\S3.4.4}.  Hence, item \eqref{fling2} of Theorem \ref{xion} immediately implies the centred statement in item (a), while for the reversal, the centred statement provides a sequence that includes all removable discontinuities, i.e.\ where $f(x)\ne f(x+)$ but $f(x+)=f(x-)$.  
By \cite{dagsamXIV}*{Theorem~2.16}, $\ACAo$ suffices to enumerate all jump discontinuities of a regulated function.   
Since there are no other discontinuities for $f$, the corollary follows.
\end{proof}
We could obtain similar results for e.g.\ Bernstein or Hermit-Fejer polynomials as analogous results hold for $BV$-functions (see \cite{samBIG3}). 
Other variations are discussed in Remark \ref{remvar} below. 

\smallskip

Fifth, as noted in Section \ref{deffer}, enumerating the points of discontinuity of a regulated function implies $\cocode_{0}$; the latter yields $\ATR_{0}$ when combined with $\ACAo$.  
By contrast, item \eqref{fling2222} in the following theorem is much weaker.  
\begin{thm}[$\ACAo+\IND_{0}$]\label{dinggong}
The higher items imply the lower ones.  
\begin{enumerate}
\renewcommand{\theenumi}{\roman{enumi}}
\item The principle $\QFAC^{0,1}$.
\item For regulated $f:[0,1]\di \R$ with infinite $D_{f}$, there is a sequence of distinct points of discontinuity of $f$.  \label{fling2222}
\item The principle $\FC$. 
\item The principle $\FSAC$.
\end{enumerate}
\end{thm}
\begin{proof}
The first downward implication is immediate by applying $\QFAC^{0,1}$ -modulo $(\exists^{2})$- to `$D_{f}$ is not finite'.  
The final implication is straightforward.  
For the second downward implication, let $(X_{n})_{n\in \N}$ be a sequence of non-empty finite sets and let $\eta:[0,1]\di \R$ be such that $\eta(x)$ is the binary expansion of $x$, choosing a tail of zeros if necessary.  
Define $h:[0,1]\di \R$ as:  
\[
h(x):=
\begin{cases}
\frac{1}{2^{n}} & \textup{ if $\eta(x)=\underbrace{11\dots 11}_{\textup{$k+1$-times}}*\langle 0 \rangle* g_{0}\oplus\dots \oplus g_{k}$ and $ (\forall i\leq k)(g_{i}\in X_{i})
$}\\
0 & \textup{otherwise} 
\end{cases}.
\]
Using $\IND_{2}$, one readily shows that $h$ is regulated (with left and right limits equal to zero) and that $D_{h}$ is infinite if $\cup_{n\in \N}X_{n}$ is.  Any sequence in $D_{h}$ then yields a sequence as in the consequent of $\FC$. 
\end{proof}
An interesting variation is provided by the following corollary.   We conjecture that $\FC$ cannot be obtained from the second item. 
\begin{cor}[$\ACAo$]
The higher items imply the lower ones.  
\begin{enumerate}
\renewcommand{\theenumi}{\roman{enumi}}
\item The principle $\QFAC^{0,1}$.
\item For $f:[0,1]\di \R$ in $BV$ with infinite $D_{f}$, there is a sequence of distinct points of discontinuity of $f$.  \label{fling22223}
\item \($\FC'$\) Let $(X_{n})_{n\in \N}$ be a sequence of non-empty finite sets in $[0,1]$ and let $g\in \N^{\N}$ be such that $|X_{n}|\leq g(n)$.  Then there is a sequence $(x_{n})_{n\in \N}$ such that $x_n\in X_{n}$ for all $n\in \N$.\label{drugz}
\item The principle $\WFSAC$.
\end{enumerate}
\end{cor}
\begin{proof}
The final implication is straightforward while the first one follows as in the proof of the theorem.  
For the second downward implication, let $(X_{n})_{n\in \N}$ be a sequence of non-empty finite sets with $|X_{n}|\leq g(n)$.  
Define $h:[0,1]\di \R$ as in the proof of the theorem but replacing `$\frac{1}{2^{n}}$' in the first case by $\frac{1}{2^{n}}\frac{1}{g(n)+1}$.  
By construction, $h$ is in $BV$ with $V_{0}^{1}(f)\leq 1$ and the set $D_{h}$ is infinite if $\cup_{n\in \N}X_{n}$ is.  Any sequence in $D_{h}$ then yields the sequence as in the consequent of $\FC'$.
\end{proof}
Finally, we discuss numerous possible variations of the above results in Section~\ref{remvar}, including Riemann integration and rectifiability.  

\section{Other topics in hyperarithmetical analysis}\label{fotehr}

\subsection{Semi-continuity and closed sets}\label{evelinenef}
We show that basic properties of semi-continuous functions, like the extreme value theorem, exist in the range of hyperarithmetical analysis.  
Since upper semi-continuous functions are closely related to closed sets, the latter also feature prominently. 

\smallskip

First of all, we need Baire's notion of semi-continuity first introduced in \cite{beren}.
\bdefi\label{flung} 
For $f:[0,1]\di \R$, we have the following definitions:
\begin{itemize}
\item $f$ is \emph{upper semi-continuous} at $x_{0}\in [0,1]$ if for any $k\in \N$, there is $N\in \N$ such that $(\forall y\in B(x_{0}, \frac{1}{2^{N}}))( f(y)< f(x_{0})+\frac{1}{2^{k}} )$. 
\item $f$ is \emph{lower semi-continuous} at $x_{0}\in [0,1]$ if for any $k\in \N$, there is $N\in \N$ such that $(\forall y\in B(x_{0}, \frac{1}{2^{N}}))( f(y)> f(x_{0})-\frac{1}{2^{k}} )$. 
\end{itemize}
\edefi
We use the common abbreviations `usco' and `lsco' for the previous notions.  We say that `$f:[0,1]\di \R$ is usco' if $f$ is usco at every $x\in [0,1]$.  
Following \cite{martino}, the extreme value theorem does not really generalise beyond semi-continuous functions.  

\smallskip

Secondly, we have the following theorem, a weaker version of which is in \cite{dagsamXVI}. 
We repeat that since the characteristic function of a closed set is usco, the connection between items \eqref{zokp} and $\CLC$ is not that surprising.  
\begin{thm}[$\ACAo+\IND_{2}$]\label{konk5}
The higher items imply the lower ones.  
\begin{enumerate}
\renewcommand{\theenumi}{\roman{enumi}}
\item The principle $\QFAC^{0,1}$.
\item \(Extreme value theorem\) For usco $f:\R\di \R$ with $y=\sup_{x\in [n, n+1]}f(x)$ for all $n\in \N$, there is $(x_{n})_{n\in \N}$ such that $(\forall n\in \N)(x_{n}\in [n, n+1]\wedge f(x_{n})=y)$.\label{zokp}  
\item  \(\CLC, \cite{dagsamXVI}\) Let $(C_{n})_{n\in \N}$ be a sequence of non-empty closed sets in $[0,1]$.  Then there is $(x_{n})_{n\in \N}$ such that $x_n\in C_{n}$ for all $n\in \N$.
\item For usco and regulated $f:\R\di \R$ with $y=\sup_{x\in [n, n+1]}f(x)$ for all $n\in \N$, there is $(x_{n})_{n\in \N}$ such that $(\forall n\in \N)(x_{n}\in [n, n+1]\wedge f(x_{n})=y)$.\label{zokp2}  
\item  \(\FC\) Let $(X_{n})_{n\in \N}$ be a sequence of non-empty finite sets in $[0,1]$.  Then there is $(x_{n})_{n\in \N}$ such that $x_n\in X_{n}$ for all $n\in \N$.
\item The principle $\FSAC$.
\end{enumerate}
\end{thm}
\begin{proof}
For the first downward implication, if the supremum $y$ is given, we have $(\forall n,k\in \N)(\exists x\in [n, n+1])( f(x)>y-\frac{1}{2^{k}}  )$, and applying $\QFAC^{0,1}$ yields a sequence $(x_{n, k})_{n,k\in \N}$.  
Since $(\exists^{2})\di \ACA_{0}$, the latter has a convergent sub-sequence (for fixed $n\in \N$), with limit say $y_{n}\in [n, n+1]$ by sequential completeness.  One readily verifies that $f(y_{n})=y$ for any $n\in \N$ as $f$ is usco.
For the second implication, fix a sequence $(C_{n})_{n\in \N}$ of closed sets and define $h:[0,1]\di \R$ as follows using Feferman's $\mu$:
\be\label{mopi}
h(x):=
\begin{cases}
1 &  x-n\in C_{n} \wedge n>0 \\
0 & \textup{ otherwise}
\end{cases}.
\ee
Since $h$ is essentially the characteristic function of closed sets, $h$ is usco on $[n, n+1]$ by definition, for each $n\in \N$. 
The sequence provided by item \eqref{zokp} then clearly satisfies $x_{n}\in C_{n}$. 
To show that $\CLC$ implies item \eqref{zokp}, let $f:[0,1]\di \R$ and $ y\in \R$ be as in the latter and define $C_{n,k}=\{x\in [n, n+1]: f(x)\geq y -\frac{1}{2^{k}}\}$ which is non-empty by definition and closed as $f$ is usco.  The sequence provided by $\CLC$ yields $x_{n}\in [n, n+1] $ such that $f(x_{n})=y$.  
The function $h$ from \eqref{mopi} is also regulated in case each $C_{n}$ is finite, i.e.\ the fourth implication also follows.  
\end{proof}
We note that item \eqref{zokp} is equivalent to e.g.\ the sequential version of the Cantor intersection theorem (\cite{dagsamXVI}).

\smallskip

Thirdly, $\CLC$ from Theorem \ref{konk5} is provable in $\WKL_{0}$ if assume that the closed sets are given by a sequence of RM-codes (see \cite{simpson2}*{IV.1.8}).
We next study $\CLC$ for an alternative representation of closed sets from \cites{browner, brownphd, browner2} as follows.
\bdefi\label{faalanx}
A \(code for a\) \emph{separably closed} set is
a sequence $S=(x_{n})_{n\in \N}$ of reals.  We write `$x\in \overline{S}$' in case $(\forall k\in \N)(\exists n\in \N)(|x-x_{n}|<\frac{1}{2^{k}}$.  
A \(code for a\) separably open set is a code for the \(separably closed\) complement.  
\edefi
Next, item \eqref{Fok} in Theorem \ref{dek} is a weakening of \cite{simpson2}*{V.4.10}, which in turn is a second-order version of the countable union theorem.
In each case, the antecedent only expresses that for every $n$, there \emph{exists} an enumeration of $A_{n}$; abusing notation\footnote{In particular, the formula `$X\in \overline{A_{n}}$' in Theorem \ref{dek} is short-hand for 
\[
(\exists (Y_{m})_{m\in \N})\big[  (\forall Y\subset \N)( Y\in A_{n}\di (\exists m\in \N)(Y_{m}= Y ) )\wedge (\forall k\in \N)(\exists l\in \N)( \overline{X}k=\overline{Y_{l}}k \wedge Y_{l}\in A_{n})  \big],
\]
which is slightly more unwieldy.} slightly, we still write `$X\in \overline{A_{n}}$' as in Definition \ref{faalanx}, leaving the enumeration of $A_{n}$ implicit.  
We sometimes identify subsets $X\subset \N$ and elements $f\in 2^{\N}$.  
\begin{thm}[$\ACA_{0}$]\label{dek}
The following items are intermediate between $\SAC$ and $\WSAC$.
\begin{enumerate}
\renewcommand{\theenumi}{\roman{enumi}}
\item Let $(A_{n})_{n\in \N}$ be a sequence of analytic codes such that each $A_{n}$ is enumerable and non-empty.  
There is a sequence $(X_{n})_{n\in \N}$ with $(\forall n\in \N)(X_{n}\in \overline{A_{n}})$.\label{Fok}
\item Let $(A_{n})_{n\in \N}$ be a sequence of analytic codes such that $A_{n}$ is enumerable and $\overline{A_{n}}$ has positive measure.  
There exists $(X_{n})_{n\in \N}$ with $(\forall n\in \N)(X_{n}\in \overline{A_{n}})$.\label{Fok3}
\item Let $(A_{n})_{n\in \N}$ be a sequence of analytic codes such that $A_{n}$ is enumerable and $\overline{A_{n}}$ is not enumerable.  
There exists $(X_{n})_{n\in \N}$ with $(\forall n\in \N)(X_{n}\in \overline{A_{n}})$.\label{Fok4}
\item Let $(A_{n})_{n\in \N}$ be a sequence of analytic codes such that for all $n\in \N$, $A_{n}$ is RM-open.  
There exists $(X_{n})_{n\in \N}$ with $(\forall n\in \N)(X_{n}\in {A_{n}})$.\label{Fok2}
\end{enumerate}
\end{thm}
\begin{proof}
To prove the items in $\SAC$, apply the latter to $(\forall n\in \N)(\exists X\subset \N)[X\in A_{n}]  $, noting that the formula in square brackets is $\Sigma_{1}^{1}$ if $A_{n}$ is an analytic code.  
To derive $\WSAC$ from item \eqref{Fok}, let $\varphi$ be arithmetical and such that $(\forall n\in \N)(\exists! X\subset \N )\varphi(X, n)$ and define `$X\in A_{n}$' as $\varphi(X, n)$ using \cite{simpson2}*{V.1.7$'$}.    Clearly, $X\in \overline{A_{n}}$ then implies $\varphi(X, n)$ as $A_{n}$ codes a singleton, i.e.\ item \eqref{Fok} implies $\WSAC$.
To obtain $\WSAC$ from item \eqref{Fok2}, let $\varphi$ be as in the antecedent of the former and consider $\Psi(X, n, k)$ defined as
\be\label{piesje}
(\exists Y\subset \N)[\varphi(Y,n )\wedge \overline{Y}k=\overline{X}k],
\ee
which yields a sequence of $\Sigma_{1}^{1}$-formulas, yielding in turn a sequence of analytic codes $(A_{n, k})_{n, k\in \N}$ by \cite{simpson2}*{V.1.7$'$}. 
In light of \eqref{piesje}, $A_{n, k}$ is a basic open ball in $2^{\N}$.  
In case $X_{n, k}\in A_{n, k}$ for all $n, k\in \N$, define $Y_{n}:= \lambda k. \overline{X_{n,k}}k$ and note that $\varphi(Y_{n}, n)$ for all $n\in \N$.
To obtain $\WSAC$ from item \eqref{Fok3}, let $\varphi$ be as in the antecedent of the former and define $\Psi(X, n, k)$ as
\[
(\exists Y\subset \N)[\varphi(\overline{X}k*Y,n )\wedge (\exists \sigma \in 2^{<\N})(X=\sigma *00\dots )],
\]
which yields a sequence of $\Sigma_{1}^{1}$-formulas, yielding in turn a sequence of analytic codes $(A_{n, k})_{n, k\in \N}$ by \cite{simpson2}*{V.1.7$'$}. 
For fixed $n_{0}\in \N$, there is a unique $X_{0}\subset \N$ such that $\varphi(X_{0}, n_{0})$, immediately yielding an enumeration of $A_{n_{0}, k}$ for any $k\in \N$.
Essentially by definition, $\overline{A_{n,k}}$ has measure $1/2^{k}$.  
In case $X_{n, k}\in \overline{A_{n, k}}$ for all $n, k\in \N$, define $Y_{n}:= \lambda k. \overline{X_{n,k}}k$ and note that $\varphi(Y_{n}, n)$ for all $n\in \N$.
Item \eqref{Fok4} also follows as enumerable sets have measure zero. 
\end{proof}
We would like to formulate item \eqref{Fok} using Borel codes from \cite{simpson2}*{V.3}, but the latter seem to need $\ATR_{0}$ to express basic aspects.  
The items from the theorem also imply $\FSAC$, which is left as an exercise.  

\smallskip

Finally, we formulate a higher-order result for comparison; we continue the abuse of notation involving $\overline{S_{n}}$ as in Theorem \ref{konk5}. 
\begin{thm}[$\ACAo$]\label{konk6}
The higher items imply the lower ones.  
\begin{enumerate}
\renewcommand{\theenumi}{\roman{enumi}}
\item The principle $\QFAC^{0,1}$.
\item Let $(S_{n})_{n\in \N}$ be a sequence of sets in $[0,1]$ such that for all $n\in \N$, $S_{n}$ is enumerable and non-empty.  There is $(x_{n})_{n\in \N}$ with $(\forall n\in \N)(x_n\in \overline{S_{n}})$.\label{zitem666}
\item Let $(S_{n})_{n\in \N}$ be a sequence of sets in $[0,1]$ such that for all $n\in \N$, $S_{n}$ is enumerable and $\overline{S_{n}}$ has positive measure.  There is $(x_{n})_{n\in \N}$ with $(\forall n\in \N)(x_n\in \overline{S_{n}})$.\label{zitem6}
\item $\cocode_{1}$
\end{enumerate}
\end{thm}
\begin{proof}
The first downward implication follows by applying $\QFAC^{0,1}$ to `$S_{n}$ is non-empty for all $n\in \N$'.
For the third downward implication, let $Y:[0,1]\di \R $ and $ A\subset [0,1]$ be such that $(\forall n\in \N)(\exists! x\in A)(Y(x)=n)$.  Define the set
\[\textstyle
E_{n,k}:=\{  x\in [n, n+1]:  (\exists q\in \Q)(Y(x-n+q)=n\wedge x-n+q\in A \wedge |q|\leq \frac{1}{2^{k+1}}) \} 
\]
and note that this sequence has a straightforward enumeration while the associated separably closed set has measure $\frac{1}{2^{k}}$.  
Let $(x_{n, k})_{n, k\in \N}$ be the sequence provided by item \eqref{zitem6}.    Using sequential compactness, $y_{n}=\lim_{k\di \infty}x_{n,k}$ is a real in $[n, n+1]$ satisfying $Y(y_{n})=n$, for any $n\in \N$ as required.
\end{proof}
Variations of the previous theorem are possible, e.g.\ replacing `enumerable' by `(strongly) countable'.
Nonetheless, we are not able to derive e.g.\ $\cocode_{1}$ from  $\CLC$ restricted to closed sets of positive measure, i.e.\ the previous two theorems may well be 
due to the coding of closed sets as in Definition \ref{faalanx}.  

\subsection{Unordered sums}\label{unorder}
The notion of \emph{unordered sum} is a device for bestowing meaning upon sums involving uncountable index sets.  
We introduce the relevant definitions and then prove that basic properties of unordered sums exist in the range of hyperarithmetical analysis.  

\smallskip

First of all, {unordered sums} are essentially `uncountable sums' $\sum_{x\in I}f(x)$ for \emph{any} index set $I$ and $f:I\di \R$.  
A central result is that if $\sum_{x\in I}f(x)$ somehow exists, it must be a `normal' series of the form $\sum_{i\in \N}f(y_{i})$, i.e.\ $f(x)=0$ for all but countably many $x\in [0,1]$; Tao mentions this theorem in \cite{taomes}*{p.~xii}. 

\smallskip

By way of motivation, there is considerable historical and conceptual interest in this topic: Kelley notes in \cite{ooskelly}*{p.\ 64} that E.H.\ Moore's study of unordered sums in \cite{moorelimit2} led to the concept of \emph{net} with his student H.L.\ Smith (\cite{moorsmidje}).
Unordered sums can be found in (self-proclaimed) basic or applied textbooks (\cites{hunterapp,sohrab}) and can be used to develop measure theory (\cite{ooskelly}*{p.\ 79}).  
Moreover, Tukey shows in \cite{tukey1} that topology can be developed using \emph{phalanxes}, which are nets with the same index sets as unordered sums.  

\smallskip

Secondly, as to notations, unordered sums are just a special kind of \emph{net} and $a:[0,1]\di \R$ is therefore written $(a_{x})_{x\in [0,1]} $ in this context to suggest the connection to nets.  
The associated notation $\sum_{x\in [0,1]}a_{x}$ is purely symbolic.   
We only need the following notions in the below. 
Let $\fin(\R)$ be the set of all finite sequences of reals without repetitions.  
\bdefi\label{kaukie} Let $a:[0,1]\di \R$ be any mapping, also denoted $(a_{x})_{x\in [0,1]}$.
\begin{itemize}
\item We say that $(a_{x})_{x\in [0,1]} $ is \emph{convergent to $a\in \R$} if for all $k\in \N$, there is $I\in \fin({\R})$ such that for $J \in \fin({\R})$ with $I\subseteq J$, we have $|a-\sum_{x\in J}a_{x}|<\frac{1}{2^{k}}$.
\item A \emph{modulus of convergence} is any sequence $\Phi^{0\di 1^{*}}$ such that $\Phi(k)=I$ for all $k\in \N$ in the previous item.  
\end{itemize}
\edefi
For simplicity, we focus on \emph{positive unordered sums}, i.e.\ $(a_{x})_{x\in [0,1]}$ such that $a_{x}\geq 0$ for $x\in [0,1]$.

\smallskip

Thirdly, we establish that basic properties of unordered sums exist in the range of hyperarithmetical analysis. 
\begin{thm}[$\ACAo+\IND_{1}$]\label{uot}
The higher items imply the lower ones.
\begin{enumerate}
\renewcommand{\theenumi}{\roman{enumi}}
\item $\QFAC^{0,1}$.
\item For a positive and {convergent} unordered sum $\sum_{x\in [0,1]}a_{x}$, there is a sequence $(y_{n})_{n\in \N}$ of reals such that $a_{y}=0$ for all $y$ not in this sequence.\label{bvv7}
\item For a positive and {convergent} unordered sum $\sum_{x\in [0,1]}a_{x}$, there is a modulus of convergence. 
\item $\cocode_{1}$.
\end{enumerate}
\end{thm}
\begin{proof}
Assume $\QFAC^{0,1}$ and note that the convergence of an unordered sum to some $a\in \R$ implies
\be\textstyle\label{furlo}
(\forall k\in \N\)(\exists I\in \fin(\R)) \big(|a-\sum_{x\in I}a_{x}|<\frac{1}{2^{k}}\big).
\ee
Apply $\QFAC^{0,1}$ to \eqref{furlo} to obtain a sequence $(I_{n})_{n\in \N}$ of finite sequences of reals.
This sequence must contain all $y\in \R$ such that $a_{y}\ne0$.  Indeed, suppose $y_{0}\in \R$ satisfies $a_{y_{0}}>_{\R}\frac{1}{2^{k_{0}}}$ for fixed $k_{0}\in \N$ and $y_{0}$ is not included in $(I_{n})_{n\in \N}$.  By definition, $I_{k_{0}+2}$ satisfies $|a-\sum_{x\in I_{k_{0}+2}}a_{x}|<\frac{1}{2^{k_{0}+2}}$.    However, for $J=I_{k_{0}+2}\cup\{y_{0}\}$, we have $a_{J}>a$, a contradiction. 
Hence, $\QFAC^{0,1}$ implies item \eqref{bvv7}.
The second and third items are readily seen to be equivalent.  

\smallskip

For the final downward application, let $A\subset [0,1]$ and $Y:[0,1]\di \R$ be such that the latter is injective and surjective on the former. 
Define $a_{x}:=\frac{1}{2^{Y(x)+1}}$ if $x\in A$, and $0$ otherwise.  One readily proves that $\sum_{x\in [0,1]}a_{x}$ is convergent to $1$, for which $\IND_{1}$ is needed.
The sequence from the second item now yields the enumeration of the set $A$ required by $\cocode_{1}$. 
\end{proof}
We note that $\WFSAC$ can be obtained from item \eqref{bvv7} in Theorem~\ref{uot}; we conjecture that $\FSAC$ cannot be obtained. 
Since unordered sums are just nets, one could study statements like
\begin{center}
\emph{a convergent net has a convergent sub-sequence},
\end{center}
which for index sets defined over Baire space is equivalent to $\QFAC^{0,1}$ (\cite{samhabil}).

\subsection{Variations and generalisations}\label{remvar}
We discuss variations and generalisations of the above results.  

\smallskip

First of all, many variations of the results in Section \ref{jefdoetbef} exist for \emph{rectifiable} functions.   
Now, Jordan proves in \cite{jordel3}*{\S105} that $BV$-functions are exactly those for which the notion of `length of the graph of the function' makes sense.  In particular, $f\in BV$ if and only if the `length of the graph of $f$', defined as follows:
\be\label{puhe}\textstyle
L(f, [0,1]):=\sup_{0=t_{0}<t_{1}<\dots <t_{m}=1} \sum_{i=0}^{m-1} \sqrt{(t_{i}-t_{i+1})^{2}+(f(t_{i})-f(t_{i+1}))^{2}  }
\ee
exists and is finite by \cite{voordedorst}*{Thm.\ 3.28.(c)}.  In case the supremum in \eqref{puhe} exists (and is finite), $f$ is also called \emph{rectifiable}.  
Rectifiable curves predate $BV$-functions: in \cite{scheeffer}*{\S1-2}, it is claimed that \eqref{puhe} is essentially equivalent to Duhamel's 1866 approach from \cite{duhamel}*{Ch.\ VI}.  Around 1833, Dirksen, the PhD supervisor of Jacobi and Heine, already provides a definition of arc length that is (very) similar to \eqref{puhe} (see \cite{dirksen}*{\S2, p.\ 128}), but with some conceptual problems as discussed in \cite{coolitman}*{\S3}.

\smallskip

Secondly, regulated functions are not necessarily $BV$ but have \emph{bounded} Waterman variation $W_{0}^{1}(f)$ (see \cite{voordedorst}), which is a generalisation of $BV$ where the sum in \eqref{tomb} 
is weighted by a \emph{Waterman sequence}, which is a sequence of positive reals that converges to zero and with a divergent series.  
Some of the above results generalise to regulated function for which the Waterman variation is known, say $W_{0}^{1}(f)=1$.

\smallskip

Thirdly, one can replace the consequent of item \eqref{fling2} in Theorem \ref{xion} by a number of similar conditions, like the existence of a Baire 1 representation (which can be defined in $\ACAo$ for monotone functions), the fundamental theorem of calculus at all reals but a given sequence, or the condition that if the Riemann integral of $f:[0,1]\di [0,1]$ in $BV$ is zero, $f(x)=0$ for all $x\in [0,1]$ but a given sequence.  Many similar conditions may be found in \cites{samBIG, samBIG3, samBIG4}.  

\smallskip

Fourth, Theorem \ref{dinggong} is readily generalised to (almost) arbitrary functions on the reals.  
To make sure the resulting theorem is provable in $\ACAo+\QFAC^{0,1}$, it seems we need \emph{oscillation functions}\footnote{For any $f:\R\di \R$, the associated \emph{oscillation functions} are defined as follows: $\osc_{f}([a,b]):= \sup _{{x\in [a,b]}}f(x)-\inf _{{x\in [a,b]}}f(x)$ and $\osc_{f}(x):=\lim _{k \di \infty }\osc_{f}(B(x, \frac{1}{2^{k}}) ).$}. Riemann, Ascoli, and Hankel already considered the notion of oscillation in the study of Riemann integration (\cites{hankelwoot, rieal, ascoli1}), i.e.\ there is ample historical precedent.   In the same way as for Theorem \ref{dinggong}, one proves that the higher items imply the lower ones over $\ACAo$.
\begin{itemize}
\item The principle $\QFAC^{0,1}$.
\item Any infinite set $X\subset [0,1]$ has a limit point. 
\item For any $f:[0,1]\di \R$ with oscillation function $\osc_{f}:[0,1]\di\R$, the set $D_{f}=\{x\in [0,1]:\osc_{f}(x)>0\}$ is either finite or has a limit point. 
\item For a non-piecewise continuous $f:[0,1]\di \R$ with oscillation function $\osc_{f}:[0,1]\di\R$, the set $D_{f}=\{x\in [0,1]:\osc_{f}(x)>0\}$ has a limit point. 
\item The arithmetical Bolzano-Weierstrass theorem $\ABW_{0}$ (\cite{coniving}). 
\end{itemize}
We note that $\osc_{f}:[0,1]\di \R$ is necessary to make `$x\in D_{f}$' into an \emph{arithmetical} formula while `$x$ is a limit point of $D_{f}$' is a meaningful (non-arithmetical) formula even if $D_{f}$ does not exist as a set. 

\begin{ack}\rm 
Our research was supported by the \emph{Klaus Tschira Boost Fund} via the grant Projekt KT43.
The initial ideas for this paper, esp.\ Section~\ref{disorder}, were developed in my 2022 Habilitation thesis at TU Darmstadt (\cite{samhabil}) under the guidance of Ulrich Kohlenbach.  
The main ideas of this paper came to the fore during the \emph{Trends in Proof Theory} workshop in February 2024 at TU Vienna.  
We express our gratitude towards all above persons and institutions.   
\end{ack}

\bye